\documentclass[12pt]{amsart}

\usepackage{ucs}

\usepackage{amssymb}
\usepackage{amsthm}
\usepackage{amsmath}
\usepackage{latexsym}
\usepackage[cp1251]{inputenc}
\usepackage{graphicx}
\usepackage{wrapfig}
\usepackage{caption}
\usepackage{subcaption}
\usepackage{indentfirst}
\usepackage[left=2.6cm,right=2.6cm,top=2.6cm,bottom=2.6cm,bindingoffset=0cm]{geometry}
\usepackage{enumerate}
\usepackage{makecell}

\DeclareMathOperator{\aut}{Aut}

\DeclareMathOperator{\cay}{Cay}
\DeclareMathOperator{\cyc}{Cyc}

\DeclareMathOperator{\id}{id}

\DeclareMathOperator{\orb}{Orb}

\DeclareMathOperator{\rk}{rk}

\DeclareMathOperator{\Span}{Span}

\DeclareMathOperator{\Syl}{Syl}
\DeclareMathOperator{\sym}{Sym}
\DeclareMathOperator{\rad}{rad}

\DeclareMathOperator{\GCD}{GCD}

\DeclareMathOperator{\Hol}{Hol}

\makeatletter 
\def\@seccntformat#1{\csname the#1\endcsname. } 
\def\@biblabel#1{#1.}

\makeatother

\title{On schurity of dihedral groups}

\author{Grigory Ryabov}
\address{School of Mathematical Sciences, Hebei Key Laboratory of Computational Mathematics and Applications, Hebei Normal University, Shijiazhuang 050024, P. R. China}

\email{gric2ryabov@gmail.com}

\thanks{The author was supported by the grant of The Natural Science Foundation of Hebei Province (project No.~A2023205045)}

\date{}

\newtheorem{prop}{Proposition}[section]

\newtheorem{lemm}[prop]{Lemma}
\newtheorem{theo}[prop]{Theorem}

\newtheorem{corl}[prop]{Corollary}

\newtheorem*{rem1}{Remark~1}

\def\tm#1{\item[{\rm (#1)}]}

\begin{document}

\vspace{\baselineskip}
\vspace{\baselineskip}

\vspace{\baselineskip}

\vspace{\baselineskip}

\begin{abstract}
A finite group $G$ is called a \emph{Schur} group if every $S$-ring over $G$ is \emph{schurian}, i.e. associated in a natural way with a subgroup of $\sym(G)$ that contains all right translations. One of the crucial questions in the $S$-ring theory is the question on schurity of nonabelian groups, in particular, on existence of an infinite family of nonabelian Schur groups. In this paper, we study schurity of dihedral groups. We show that any generalized dihedral Schur group is dihedral and obtain necessary conditions of schurity for dihedral groups. Further, we prove that a dihedral group of order~$2p$, where $p$ is a Fermat prime or prime of the form $p=4q+1$, where $q$ is also prime, is Schur. Towards this result, we prove nonexistence of a difference set in a cyclic group of order~$p\neq 13$ and classify all $S$-rings over some dihedral groups.


\end{abstract}

\maketitle

\section{Introduction}

A \emph{Schur ring} or \emph{$S$-ring} over a finite group $G$ can be defined as a subring of the group ring $\mathbb{Z}G$ that is a free $\mathbb{Z}$-module spanned by a partition of $G$ closed under taking inverse and containing the identity element $e$ of $G$ as a class. The theory of $S$-rings was initiated by Schur~\cite{Schur} and developed later by Wielandt~\cite{Wi}. Schur and Wielandt used $S$-rings to study permutation groups containing regular subgroups. To the present moment, $S$-rings have become a useful tool for studying permutation groups and Cayley graphs, in particular, the Cayley graph isomorphism problem. On the other hand, the theory of $S$-rings continues to develop itself. Concerning the latter one, we refer the readers to the survey~\cite{MP0}.

One of the keynote properties of $S$-rings is \emph{schurity}. If $K$ is a subgroup of the symmetric group $\sym(G)$ containing the group $G_r$ of all right translations, then the partition of $G$ into the orbits of the stabilizer~$K_e$ defines an $S$-ring over~$G$. As Wielandt wrote~\cite{Wi2}, Schur had conjectured that every $S$-ring can be constructed in such way from a suitable permutation group. However, this conjecture was disproved by Wielandt~\cite{Wi}. Later, $S$-rings arising from permutation groups were called \emph{schurian}, whereas groups whose all $S$-rings are schurian were called \emph{Schur} groups~\cite{Po}. The schurity property for $S$-rings is closely related to the isomorphism poblem for Cayley graphs (see~\cite{CP,MP0}). The problem of determining of all Schur groups was posed by P\"{o}schel~\cite{Po}. In general, this problem seems hard. One of the reasons for this is that to verify schurity of a group $G$, it is needed to describe all $S$-rings over $G$, but the number of $S$-rings over $G$ is exponential in the order of $G$.

Most of the results on Schur groups are concerned with the abelian case. They can be found in~\cite{EKP1,EKP2,GNP,KP,MP,PR,Po,Ry2,Ry4}. In fact, these results imply a complete classification of abelian Schur groups. However, since the general theory of $S$-rings over nonabelian groups is almost undeveloped, there are only few results on schurity of nonabelian groups. Necessary conditions of schurity for nonabelian groups were obtained in~\cite{PV}. In~\cite{Ry3}, it was shown that every nilpotent Schur group belongs to one of the explicitly given families of groups. Nonschurity of some nonabelian groups was checked in~\cite{MP,Ry1}. There are several examples of nonabelian Schur groups. The largest known nonabelian Schur group of order~$63$ was found by computer calculations~\cite{Ziv}.

It seems reasonable and interesting to study schurity of different classes of nonabelian groups. In the present paper, we deal with the class of dihedral groups. On the one hand, studying schurity of dihedral groups looks as the most natural step after the abelian case towards the general one. On the other hand, the schurity problem for dihedral groups is closely related to another important problem of the algebraic combinatorics, namely, to the problem of existence of a difference set in a cyclic group. This connection will be discussed further. Apparently, there are only two papers~\cite{MP} and~\cite{PV} which are concerned with schurity of dihedral groups. In~\cite{MP}, all $S$-rings of rank at most~$5$ over dihedral $2$-groups were classified and it was shown that only $S$-rings associated with divisible difference sets can be nonschurian among them, whereas in ~\cite{PV}, it was proved that a dihedral group of order~$2p$, where $p$ is a prime, is not Schur whenever $p\equiv 3\mod~4$ and $p>11$.

Recall that a nonabelian group is said to be \emph{generalized dihedral} if it is a semidirect product of an abelian group and a cyclic group of order~$2$, where the nontrivial element of the latter one inverses every element of the former one. The first main result of the paper is the theorem below.

\begin{theo}\label{main1}
Every generalized dihedral Schur group is dihedral.
\end{theo}

The next result of the paper provides restrictions on the order of a dihedral Schur group. The set of prime divisors of an integer $n$ is denoted by $\pi(n)$.

\begin{theo}\label{main2}
Let $n\geq 3$ be a positive integer. Suppose that a dihedral group of order~$2n$ is Schur. Then $|\pi(n)|\leq 3$. Moreover, 
\begin{enumerate}

\tm{1} if $|\pi(n)|=2$, then $n$ is a product of a prime and an odd prime power or $n=12$,

\tm{2} if $|\pi(n)|=3$, then $n$ is a product of three primes.  

\end{enumerate}
\end{theo}

To prove Theorems~\ref{main1} and~\ref{main2}, we construct new families of nonschurian $S$-rings and prove that some known nonschurian $S$-rings are isomorphic to $S$-rings over dihedral groups. Some of the constructions use difference sets in cyclic groups (see, e.g, Proposition~\ref{pdiff}). More specific restrictions on the order of a dihedral Schur group are given in Corollary~\ref{conditions}. Due to computational results~\cite{Ziv}, all dihedral groups of order at most~$62$ satisfying the conditions from Theorem~\ref{main2} and Corollary~\ref{conditions} are Schur except for the groups $D_{52}$ and $D_{60}$ which are not Schur and discussed in detail in Section~9.1.

As it was mentioned before, the order of the largest known nonabelian Schur group is~$63$. This leads us to the question on existence of an infinite family of nonabelian Schur groups which was posed for the first time by Olshansky at the conference ``Algebra and combinatorics'' (Ekaterinburg, 2013). Attempts to find the required family meet some problems which are hard themselves in general, e.g. the problems of existence of a (partial, divisible) difference set, computation of cyclotomic numbers, description of primitive nonschurian $S$-rings etc. One more of the main results of this paper is the following theorem which is the first result stating an existence of a potentially infinite family of nonabelian Schur groups.

\begin{theo}\label{main3}
Let $p$ be a prime. If $p$ is a Fermat prime or $p=4q+1$, where $q$ is a prime, then a dihedral group of order~$2p$ is Schur.
\end{theo}

Towards the proof of Theorem~\ref{main3}, we classify all $S$-rings over a dihedral group of order~$2p$, where $p$ is a Fermat prime, i.e. a prime of the form $p=2^k+1$, where $k\geq 1$, or $p=rq+1$ for a prime $q$ and $r\in\{2,4\}$ (Theorem~\ref{classification}). Recall that primes of the form~$2q+1$, where $q$ is also prime, are known as \emph{safe} primes (see, e.g.,~\cite{Rib}). One of the keynote cases in the proof of Theorem~\ref{classification} is the case when Cayley graphs corresponding to some basic sets of an $S$-ring are isomorphic to a Paley graph. To handle this case, we use a recent result of Brouwer and Martin~\cite{BM} on triple intersection numbers of a Paley graph (see the proof of Lemma~\ref{4q1}).

It is not known whether there are infinitely many Fermat primes or primes of the form $4q+1$, where $q$ is also prime. Only five Fermat primes are known; the largest of them is~$65537$. The question whether a number of primes of the form $4q+1$ is infinite is closely related to the famous and widely-believed Dickson, Bateman-Horn, and generalized Hardy-Littlewood conjectures (see~\cite{BH,Di0,GT2,JZ,Rib}). Namely, if one of these conjectures is true, then there are infinitely many primes of the form~$4q+1$. Computational results on primes~$q$ such that $4q+1$ is also prime can be found on the web-page~\cite{oeis}.

One of the crucial questions which appears during the studying of schurity of dihedral groups is the question on existence of a difference set in a cyclic group (see~\cite{MP,PV}). A subset $D$ of a group $G$ is called a \emph{difference set} if there is a nonnegative integer $\lambda$ such that every nonidentity element $g\in G$ has exactly $\lambda$ representations in the form~$g=d_1^{-1}d_2$, $d_1,d_2\in D$. Difference sets were introduced by Singer~\cite{Sin} and systematic studying of difference sets was started with the papers of Bruck~\cite{Bruck} and Hall~\cite{Hall}. For details on difference sets, we refer the readers to~\cite[Chapter~VI]{BJL}. The keynote ingredient in the proof of Theorem~\ref{main1} is the following theorem.

\begin{theo}\label{main4}
Let $p$ be a prime. If $p$ is a Fermat prime or $p=4q+1$, where $q>3$ is a prime, then there is no a nontrivial difference set in a cyclic group of order~$p$.
\end{theo}

Note that Theorem~\ref{main2} for a Fermat prime~$p$ immediately follows from a simple counting argument, whereas the proof of this theorem for $p=4q+1$ requires some technique of working with multipliers of difference sets. Namely, in the latter case a group $M$ of multipliers of a nontrivial difference set in a cyclic group $C_p$ of order~$p$ is a subgroup of $\aut(C_p)\cong C_{4q}$. We consider all possibilities for the order of $M$. To obtain a contradiction, we use relations between cyclotomic numbers from the classical paper of Dickson~\cite{Di} and a recent result of Gordon and Schmidt~\cite{GS} that guarantees an existence of a nontrivial multiplier. The condition $q>3$ in Theorem~\ref{main2} is essential because there is a nontrivial difference set in a cyclic group of order~$13$ (see Remark~$1$).

We finish the introduction with a brief outline of the paper. Section~$2$ contains a necessary background of $S$-rings. In Section~$3$, we provide basic definitions and facts concerned with difference sets and prove Theorem~\ref{main4}. The proof of Theorem~\ref{main1} is given in Section~$4$. Section~$5$ contains a construction of a nonschurian $S$-ring over a group of order~$2pn$, where $p$ is a prime and $n\geq 3$, which is used in the proof of Theorem~\ref{main2}. In Section~$6$, we provide several restrictions on the structure of a dihedral Schur group and prove Theorem~\ref{main2}. In Section~$7$, we classify all $S$-rings over a dihedral group of order~$2p$, where $p$ is a Fermat prime or $p=rq+1$ for a prime $q$ and $r\in\{2,4\}$, and prove Theorem~\ref{main3}. Section~$8$ contains some further remarks, in particular, on $S$-rings from relative difference sets and their connection with $2$-arc transitive dihedrants and necessary conditions of schurity for Frobenius groups. Throughout the paper, cyclic, dihedral, and elementary abelian groups of order~$n$ are denoted by $C_n$, $D_n$, and $E_n$, respectively.

The author is very grateful to Dr. Matan Ziv-Av for the help with computer calculations and Prof. Ilia Ponomarenko for the valuable comments on the text.

\section{$S$-rings}

In this section, we follow~\cite{MP} in general. Let $G$ be a finite group and $\mathbb{Z}G$ the integer group ring. The identity element and the set of all non-identity elements of $G$ are denoted by~$e$ and $G^\#$, respectively. The symmetric group of the set $G$ is denoted by~$\sym(G)$. If $K\leq \sym(G)$ and $g\in G$, then the set of all orbits of $K$ on $G$ and the one-point stabilizer of $g$ in $K$ are denoted by $\orb(K,G)$ and $K_g$, respectively. For a set $\Delta\subseteq \sym(G)$ and a section $S=B/A$ of $G$, we set 
$$\Delta^S=\{f^S:~f\in \Delta,~S^f=S\},$$
where $S^f=S$ means that $f$ permutes the $A$-cosets in $B$ and $f^S$ denotes the bijection of $S$ induced by $f$.

Given $g\in G$, the permutation of $G$ induced by the right multiplication on~$g$ is denoted by~$g_r$. The subgroup of $\sym(G)$ induced by the right multiplications of $G$ is denoted by $G_{r}$. The holomorph $G_r\rtimes \aut(G)$ of $G$ is denoted by $\Hol(G)$. If $X\subseteq G$, then the element $\sum \limits_{x\in X} {x}$ of the group ring $\mathbb{Z}G$ is denoted by~$\underline{X}$. The set $\{x^{-1}:x\in X\}$ is denoted by $X^{-1}$.

A subring  $\mathcal{A}\subseteq \mathbb{Z} G$ is called an \emph{$S$-ring} (a \emph{Schur} ring) over $G$ if there exists a partition $\mathcal{S}=\mathcal{S}(\mathcal{A})$ of~$G$ such that:

$(1)$ $\{e\}\in\mathcal{S}$;

$(2)$  if $X\in\mathcal{S}$, then $X^{-1}\in\mathcal{S}$;

$(3)$ $\mathcal{A}=\Span_{\mathbb{Z}}\{\underline{X}:\ X\in\mathcal{S}\}$.

\noindent The elements of $\mathcal{S}$ are called the \emph{basic sets} of $\mathcal{A}$. The number of basic sets is called the \emph{rank} of $\mathcal{A}$ and denoted by~$\rk(\mathcal{A})$. Clearly, $\mathbb{Z}G$ is an $S$-ring. The partition $\{\{e\},G\setminus\{e\}\}$ defines the $S$-ring $\mathcal{T}_G$ of rank~$2$ over~$G$. It is easy to check that if $X,Y\in \mathcal{S}(\mathcal{A})$, then $XY\in \mathcal{S}(\mathcal{A})$ whenever $|X|=1$ or $|Y|=1$. The $S$-ring $\mathcal{A}$ is called \emph{symmetric} if $X=X^{-1}$ for every $X\in \mathcal{S}(\mathcal{A})$.

A set $T \subseteq G$ is called an \emph{$\mathcal{A}$-set} if $\underline{T}\in \mathcal{A}$. If $T$ is an $\mathcal{A}$-set, then put $\mathcal{S}(\mathcal{A})_T=\{X\in \mathcal{S}(\mathcal{A}):~X\subseteq T\}$. A subgroup $A \leq G$ is called an \emph{$\mathcal{A}$-subgroup} if $A$ is an $\mathcal{A}$-set. The $S$-ring $\mathcal{A}$ is called \emph{primitive} if there are no nontrivial proper $\mathcal{A}$-subgroups of $G$ and \emph{imprimitive} otherwise. One can verify that for every $\mathcal{A}$-set $X$, the groups $\langle X \rangle$ and $\rad(X)=\{g\in G:\ gX=Xg=X\}$ are $\mathcal{A}$-subgroups.

Let $\{e\}\leq A \unlhd B\leq G$. A section $B/A$ is called an \emph{$\mathcal{A}$-section} if $B$ and $A$ are $\mathcal{A}$-subgroups. If $S=B/A$ is an $\mathcal{A}$-section, then the module
$$\mathcal{A}_S=Span_{\mathbb{Z}}\left\{\underline{X}^{\pi}:~X\in\mathcal{S}(\mathcal{A}),~X\subseteq B\right\},$$
where $\pi:B\rightarrow B/A$ is the canonical epimorphism, is an $S$-ring over $S$.

Let $X,Y\in\mathcal{S}$. If $Z\in \mathcal{S}$, then the number of distinct representations of $z\in Z$ in the form $z=xy$ with $x\in X$ and $y\in Y$ does not depend on the choice of $z\in Z$. Denote this number by $c^Z_{XY}$. One can see that $\underline{X}\cdot\underline{Y}=\sum \limits_{Z\in \mathcal{S}(\mathcal{A})}c^Z_{XY}\underline{Z}$. Therefore the numbers  $c^Z_{XY}$ are the structure constants of $\mathcal{A}$ with respect to the basis $\{\underline{X}:\ X\in\mathcal{S}\}$. The $S$-ring $\mathcal{A}$ is called \emph{commutative} if $c_{XY}^Z=c_{YX}^Z$ for all $X,Y,Z\in \mathcal{S}(\mathcal{A})$. If $\mathcal{A}$ is symmetric, then $\mathcal{A}$ is commutative. The equalities 
\begin{equation}\label{triangle}
|Z|c^{Z^{-1}}_{XY}=|X|c^{X^{-1}}_{YZ}=|Y|c^{Y^{-1}}_{ZX},
\end{equation}

\begin{equation}\label{sum10}
\sum \limits_{Y\in \mathcal{S}(\mathcal{A})} c_{XY}^Z=|X|,
\end{equation}
hold for all $X,Y,Z\in \mathcal{S}(\mathcal{A})$ (see~\cite[Proposition~2.1.12]{CP}).

Let $\mathcal{A}$ and $\mathcal{A}^\prime$ be $S$-rings over groups $G$ and $G^\prime$, respectively. A bijection $\varphi$ from $\mathcal{S}(\mathcal{A})$ to $\mathcal{S}(\mathcal{A}^\prime)$ is called an \emph{algebraic isomorphism} from $\mathcal{A}$ to $\mathcal{A}^\prime$ if $c_{X^\varphi Y^\varphi}^{Z^\varphi}=c_{XY}^Z$ for all $X,Y,Z\in \mathcal{S}(\mathcal{A})$. An algebraic isomorphism from $\mathcal{A}$ to itself is called an \emph{algebraic automorphism} of $\mathcal{A}$. The following statement is known as the first Schur theorem on multipliers (see~\cite[Theorem~23.9, (a)]{Wi}).

\begin{lemm} \label{burn}
Let $\mathcal{A}$ be an $S$-ring over an abelian group  $G$. Then the mapping $X\mapsto X^{(m)}=\{x^m:~x\in X\}$, $X\in \mathcal{S}(\mathcal{A})$, is an algebraic automorphism of $\mathcal{A}$ for every  $m\in \mathbb{Z}$ coprime to~$|G|$.
\end{lemm}

A bijection $f$ from $G$ to $G^\prime$ is called a (\emph{combinatorial}) \emph{isomorphism} from $\mathcal{A}$ to $\mathcal{A}^\prime$ if for every $X\in \mathcal{S}(\mathcal{A})$ there is $X^{\prime}\in \mathcal{S}(\mathcal{A}^\prime)$ such that 
$$\{(x^f,y^f)\in G\times G:~yx^{-1}\in X\}=\{(x^\prime,y^\prime):~y^\prime(x^\prime)^{-1}\in X^\prime\}.$$
If there is an isomorphism from $\mathcal{A}$ to $\mathcal{A}^\prime$, we say that $\mathcal{A}$ and $\mathcal{A}^\prime$ are \emph{isomorphic}.

A bijection $f\in\sym(G)$ is defined to be a \emph{(combinatorial) automorphism} of $\mathcal{A}$ if for all $x,y\in G$, the basic sets containing the elements $yx^{-1}$ and $y^f(x^{-1})^f$ coincide. The set of all automorphisms of $\mathcal{A}$ forms a group called the \emph{automorphism group} of $\mathcal{A}$ and denoted by $\aut(\mathcal{A})$. One can see that $\aut(\mathcal{A})\geq G_r$ and $\mathcal{A}$ is isomorphic to an $S$-ring over a group~$H$ if and only if $\aut(\mathcal{A})$ has a regular subgroup isomorphic to~$H$. The $S$-ring $\mathcal{A}$ is said to be \emph{normal} if $G_r$ is normal in $\aut(\mathcal{A})$. If $f\in \aut(\mathcal{A})$, then 
$$(Xy)^f=Xy^f$$
for all $X\in \mathcal{S}(\mathcal{A})$ and $y\in G$. If $A$ is an $\mathcal{A}$-subgroup of $G$, then the set of all right $A$-cosets is an imprimitivity system of~$\aut(\mathcal{A})$.

A group isomorphism $\alpha$ from $G$ to $G^\prime$ is defined to be a \emph{Cayley isomorphism} from $\mathcal{A}$ to $\mathcal{A}^\prime$ if $X^\alpha\in \mathcal{S}(\mathcal{A}^\prime)$ for every $X\in \mathcal{S}(\mathcal{A})$. A group automorphism $\alpha\in \aut(G)$ is defined to be a \emph{Cayley automorphism} of $\mathcal{A}$ if $X^\alpha=X$ for every $X\in \mathcal{S}(\mathcal{A})$. In this case, $\alpha \in \aut(\mathcal{A})$.

Let $K$ be a subgroup of $\sym(G)$ containing $G_{r}$. The $\mathbb{Z}$-submodule
$$V(K,G)=\Span_{\mathbb{Z}}\{\underline{X}:~X\in \orb(K_e,~G)\}$$
is an $S$-ring over $G$ as it was proved by Schur~\cite{Schur}. An $S$-ring $\mathcal{A}$ over  $G$ is called \emph{schurian} if $\mathcal{A}=V(K,G)$ for some $K\leq \sym(G)$ with $K\geq G_{r}$. One can verify that $\mathcal{A}$ is schurian if and only if
$$\mathcal{A}=V(\aut(\mathcal{A}),G).$$
It is easy to see that $\mathcal{T}_G=V(\sym(G),G)$ and hence $\mathcal{T}_G$ is schurian. Clearly, two isomorphic $S$-rings are schurian or not simultaneously. If $\mathcal{A}$ is a schurian $S$-ring and $S$ is an $\mathcal{A}$-section, then 
$$\mathcal{A}_S=V(\aut(\mathcal{A})^S,S)$$
and hence $\mathcal{A}_S$ is also schurian. A finite group~$G$ is said to be \emph{Schur} if every $S$-ring over $G$ is schurian. The above discussion implies the following lemma.

\begin{lemm}\label{schursection}
A section of a Schur group is Schur. 
\end{lemm}


Let $K \leq \aut(G)$. The set $\orb(K,G)$ forms a partition of $G$ that defines an $S$-ring $\mathcal{A}$ over~$G$. In this case,  $\mathcal{A}$ is called \emph{cyclotomic} and denoted by $\cyc(K,G)$. It is easy to see that $\cyc(K,G)=V(G_rK,G)$ and hence every cyclotomic $S$-ring is schurian. The first statement of the next lemma follows from the results of~\cite{Po} (see also~\cite[Lemma~2.4]{PR}), whereas the second one from Lemma~\ref{burn}.

\begin{lemm}\label{cyclprime}
Let $\mathcal{A}$ be an $S$-ring over a group of prime order. Then
\begin{enumerate}
\tm{1} $\mathcal{A}$ is cyclotomic and hence schurian; 

\tm{2} all nontrivial basic sets of $\mathcal{A}$ are of the same size.
\end{enumerate}
\end{lemm}

Let $\mathcal{A}$ be an $S$-ring over a group $G$ and $K\leq\aut(G)$ a group of Cayley isomorphisms from $\mathcal{A}$ to itself. Then due to~\cite[Lemma~2.3.26]{CP}, the module $\mathcal{A}^K=\Span_{\mathbb{Z}}\{\underline{X}^K:~X\in \mathcal{S}(\mathcal{A})\}$, where 
$$X^K=\bigcup \limits_{\alpha\in K} X^\alpha,$$
is an $S$-ring over $G$.

Let $S=B/A$ be an $\mathcal{A}$-section of $G$. The $S$-ring~$\mathcal{A}$ is called the \emph{$S$-wreath product} or \emph{generalized wreath product} of $\mathcal{A}_B$ and $\mathcal{A}_{G/A}$ if $A\trianglelefteq G$ and every basic set $X$ of $\mathcal{A}$ outside~$B$ is a union of some $A$-cosets or, equivalently, $A\leq \rad(X)$ for every $X\in \mathcal{S}(\mathcal{A})_{G\setminus B}$. In this case, we write $\mathcal{A}=\mathcal{A}_B \wr_S \mathcal{A}_{G/A}$. The $S$-wreath product is called \emph{nontrivial} or \emph{proper} if $A\neq \{e\}$ and $B\neq G$. If $A=B$, then the $S$-wreath product coincides with the \emph{wreath product} $\mathcal{A}_A\wr \mathcal{A}_{G/A}$ of $\mathcal{A}_A$ and $\mathcal{A}_{G/A}$. 

Given a section $S=B/A$ of a group $G$ such that $A\trianglelefteq G$ and $S$-rings $\mathcal{A}_1$ and $\mathcal{A}_2$ over $B$ and $G/A$, respectively, such that $S$ is both an $\mathcal{A}_1$- and an $\mathcal{A}_2$-section, and $(\mathcal{A}_1)_S=(\mathcal{A}_2)_S$, there is a unique $S$-ring $\mathcal{A}$ over $G$ that is the $S$-wreath product with $\mathcal{A}_B=\mathcal{A}_1$ and $\mathcal{A}_{G/A}=\mathcal{A}_2$ (see~\cite{EP0}).

Let $A$ and $B$ be subgroups of $G$ such that $G=A\times B$ and $\mathcal{A}_1$ and $\mathcal{A}_2$  $S$-rings over $A$ and $B$, respectively. Then the partition 
$$\{X_1\times X_2:~X_1\in\mathcal{S}(\mathcal{A}_1),~X_2\in \mathcal{S}(\mathcal{A}_2)\}.$$
of $G$ defines the $S$-ring $\mathcal{A}$ over $G$ called the \emph{tensor product} of $\mathcal{A}_1$ and $\mathcal{A}_2$ and denoted by $\mathcal{A}_1\otimes \mathcal{A}_2$. The tensor product is called \emph{nontrivial} if $\{e\}<A<G$ and $\{e\}<B<G$.

\section{Difference sets}

A subset $D$ of $G$ is called a \emph{difference set} in $G$ if 
$$\underline{D}\cdot\underline{D}^{-1}=ke+\lambda \underline{G}^\#,$$
where $k=|D|$ and $\lambda$ is a positive integer. The numbers $(v,k,\lambda)$, where $v=|G|$, are called the \emph{parameters} of $D$. It is easy to check that $G\setminus D$ is a difference set with parameters $(v,v-k,v-2k+\lambda)$. The empty set, the set~$G$, and every subset of $G$ of size~$1$ or~$|G|-1$ are difference sets. The difference set $D$ is called \emph{nontrivial} if $2\leq|D|\leq |G|-2$. A simple counting argument implies that
\begin{equation}\label{diff}
k(k-1)=(v-1)\lambda.
\end{equation}
For the general theory of difference sets, we refer the readers to~\cite{Bau} and~\cite[Chapter~VI]{BJL}. The lemma below provides two families of difference sets in cyclic groups. In the present paper, only an existence of such difference sets in cyclic groups is essential for us. A detailed description of the constructions can be found, e.g., in~\cite[Chapter~VI]{BJL}.

\begin{lemm}\label{diffcycl}
Let $A$ be a cyclic group of order~$n$.
\begin{enumerate}

\tm{1} If $n$ is a prime such that $n\equiv 3\bmod 4$, then $A$ has a Paley difference set with parameters~$(n,\frac{n-1}{2},\frac{n-3}{4})$.

\tm{2} If $n=\frac{q^{d+1}-1}{q-1}$ for some prime power~$q$ and positive integer~$d\geq 2$, then $A$ has a Singer difference set with parameters $(\frac{q^{d+1}-1}{q-1},\frac{q^{d}-1}{q-1},\frac{q^{d-1}-1}{q-1})$.

\end{enumerate}
\end{lemm}

\begin{lemm}\label{dif2q}
Let $p$ be a prime of the form $p=2q+1$, where $q$ is an odd prime, and $D$ a nontrivial difference set in~$C_p$. Then $D$ has  parameters $(2q+1,q,\frac{q-1}{2})$ or $(2q+1,q+1,\frac{q+1}{2})$.
\end{lemm}

\begin{proof}
Eq.~\eqref{diff} implies that $k(k-1)=2q\lambda$. Therefore $k$ or $k-1$ is divisible by~$q$. Since $D$ is nontrivial, we obtain that $k\in\{q,q+1\}$. Now the required follows from Eq.~\eqref{diff}.
\end{proof}

Note that $p\equiv 3\bmod 4$ in Lemma~\ref{dif2q} because $q$ is odd.

Let $G$ be abelian. Following~\cite{GS}, we say that an integer $t$ coprime to~$v$ is a (\emph{numerical}) \emph{multiplier} of $D$ if $D^{(t)}=\{d^t:~d\in D\}=Dg=\{dg:~d\in D\}$ for some $g\in G$. The set of automorphisms $\sigma_t:g\rightarrow g^t$, $g\in G$, where $t$ runs over all multipliers of $D$, forms a subgroup of $\aut(G)$ denoted by $M(D)$. Further for convenience, we will use the term ``multiplier'' to call the both integer~$t$ and automorphism $\sigma_t$. For more information on multipliers of difference sets, we refer the readers to~\cite{GS}.

\begin{lemm}\cite[Theorem~3.3]{Bau}\label{minus1}
The integer~$-1$ is never a multiplier of a nontrivial difference set in a cyclic group.
\end{lemm}

\begin{proof}[Proof of Theorem~\ref{main4}]
Let $A$ be a cyclic group of order~$p$ and $a$ a generator of~$A$. Assume that there is a nontrivial difference set $D$ in~$A$. Let $(p,k,\lambda)$ be the parameters of $D$. We may assume that $k<\frac{p}{2}$. Indeed, for otherwise $A\setminus D$ is a nontrivial difference set with $|A\setminus D|<\frac{p}{2}$ and we replace $D$ by $A\setminus D$. If $p$ is a Fermat prime, then $p-1$ is a $2$-power. Since $k$ and $k-1$ are coprime, Eq.~\eqref{diff} yields that $k=p$ or $k=p-1$, a contradiction to nontriviality of~$D$.

Now let $p=4q+1$, where $q>3$ is a prime.

\begin{lemm}\label{parameters}
In the above notations, $D$ has parameters $(4q+1,q,\frac{q-1}{4})$ or $(4q+1,q+1,\frac{q+1}{4})$. 
\end{lemm}

\begin{proof}
Eq.~\eqref{diff} implies that $k(k-1)=4q\lambda$. Therefore $k$ or $k-1$ is divisible by~$q$. Since $D$ is nontrivial and $k<\frac{p}{2}$, we conclude that
$$k\in\{q,q+1,2q\}.$$
If $k=2q$, then $2q-1=2\lambda$ by Eq.~\eqref{diff}. We obtain a contradiction because $2q-1$ is odd, whereas $2\lambda$ is even. Thus, $k=q$ or $k=q+1$. Now from Eq.~\eqref{diff} it follows that $\lambda=\frac{q-1}{4}$ in the former case and $\lambda=\frac{q+1}{4}$ in the latter one.
\end{proof}

Let $M=M(D)$. Clearly, $M\leq \aut(A)\cong C_{p-1}=C_{4q}$ and hence $|M|$ divides $4q$. Suppose that $|M|$ is even. Then $M$ contains an automorphism of order~$2$, i.e. $\sigma_0:a\rightarrow a^{-1}\in \aut(A)$. However, this contradicts to Lemma~\ref{minus1}. Therefore, $|M|$ is odd and hence 
$$|M|\in\{1,q\}.$$

Suppose that $|M|=1$, i.e. $M$ is trivial. Let $t$ be a positive integer such that $t<p$ and the automorphism $\sigma_t:a\rightarrow a^t$ of $A$ is of order~$q$. We are going to prove that $t$ is a nontrivial multiplier of $D$ and by that to obtain a contradiction. Let $n=k-\lambda$. From Lemma~\ref{parameters} it follows that $n>1$ and $n<p$. The latter yields that $\GCD(p,n)=1$. Let $u$ be a prime divisor of~$n$. There exists an automorphism $\sigma_u$ of $A$ such that $\sigma_u:a\rightarrow a^u$ because $u\leq n<p$. If $|\sigma_u|$ is even, then the group $\langle \sigma_u \rangle$ contains $\sigma_0:a\rightarrow a^{-1}$ and hence $u^i\equiv -1\bmod p$ for some integer $i$. Therefore $u$ is self-conjugate modulo~$p$ in the sense of~\cite[p.~7]{GS}. Suppose that $|\sigma_u|$ is odd. Then $|\sigma_u|=q$ because $|\sigma_u|$ divides $|\aut(A)|=p-1=4q$. Together with $|\sigma_t|=q$, this implies that $t\equiv u^i\bmod p$ for some integer $i$. Finally, observe that  
$$\frac{n}{\GCD(p,n)}=n=k-\lambda>\lambda$$
because $\lambda<\frac{k}{2}$ (Lemma~\ref{parameters}). Thus, all conditions of~\cite[Theorem~3.1]{GS} hold for $n_1=n$ and~$t$ and hence $t$ is a multiplier of~$D$.

Now suppose that $|M|=q$. From~\cite[VI.Theorem~2.6]{BJL} it follows that there exists a translate of $D$ which is $M$-invariant. Each translate of $D$ is also a difference set in $A$ with the same parameters as~$D$ because $A$ is abelian. So without loss of generality, we may assume that $D$ is $M$-invariant. This means that $D$ is a union of some orbits of $M$. Each non-identity orbit of $M$ is a cyclotomic class of order~$4$ (see~\cite[p.~357]{BJL} for the definition). Recall that $|D|=k\in\{q,q+1\}$ by Lemma~\ref{parameters}. If $|D|=q$, then $D$ is a non-identity orbit of $M$; if $|D|=q+1$, then $D$ is a union of a non-identity orbit of $M$ and $\{e\}$.

The number $p$ can be presented in the form
\begin{equation}\label{prime} 
p=x^2+4y^2, 
\end{equation}
where $x$ and $y$ are integers, by~\cite[Eq.~(51)]{Di}. Due to~\cite[Eq.~(55)]{Di}, we have 
\begin{equation}\label{x}
x=2q-1-8(1,0)_4,
\end{equation}
where $(1,0)_4$ is a cyclotomic number modulo~$4$ (see~\cite{Di} for the definition).

If $|D|=q$, then $(1,0)_4=\frac{q-1}{4}$ by~\cite[VI.Theorem~8.10(a)]{BJL}. Together with Eq.~\eqref{x}, this implies that $x=2q-1-8(1,0)_4=2q-1-2q+2=1$. Therefore $p=4q+1=4y^2+1$ by Eq.~\eqref{prime} and hence $q=y^2$, a contradiction to primality of~$q$. If $|D|=q+1$, then $(1,0)_4=\frac{q+1}{4}$ by~\cite[VI.Theorem~8.10(b)]{BJL}. Together with Eq.~\eqref{x}, this implies that $x=2q-1-8(1,0)_4=2q-1-2q-2=-3$. Eq.~\eqref{prime} yields that $p=4q+1=4y^2+9$ and hence $q=y^2+2$. If $y^2\equiv 1\bmod 3$, then $q$ is divisible by~$3$. So $q=3$, a contradiction to the assumption of the theorem. If $y^2$ is divisible by~$3$, then $q\equiv 2\bmod 3$. Therefore $p=4q+1$ is divisible by~$3$, a contradiction to primality of~$p$.  
\end{proof}

\begin{rem1}
If $q=3$, then there is a difference set in~$A$ with parameters~$(4q+1,q+1,\frac{q+1}{4})=(13,4,1)$ which is a Singer difference set (see Lemma~\ref{diffcycl}(2)).
\end{rem1}

Let $G$ be a dihedral group, $A$ a cyclic subgroup of $G$ of index~$2$, and $b$ an involution from $G\setminus A$. Suppose that $D$ is a difference set in $A$. Then~\cite[Lemma~5.1]{PV} implies that the partition of $G$ into the sets
$$\{e\}, A^\#, Db, (A\setminus D)b$$
defines the $S$-ring $\mathcal{A}=\mathcal{A}(D)$. Clearly, $\rk(\mathcal{A}(D))=4$ and $\mathcal{A}(D)=\mathcal{A}(A\setminus D)$.

\begin{lemm}\label{autmult}
In the above notations, $\alpha^A$ is a multiplier of $D$ for every $\alpha\in\aut(\mathcal{A})$ such that $A^\alpha=A$ and $\alpha^A\in \aut(A)$.
\end{lemm}

\begin{proof}
Due to the definitions, we have $D^\alpha=(Dbb)^\alpha=Dbb^\alpha=Da$ for some $a\in A$ as required.
\end{proof}

\section{Proof of Theorem~\ref{main1}}

Let $G$ be a generalized dihedral Schur group associated with an abelian group~$A$. Due to Lemma~\ref{schursection}, the group $A$ is Schur. So $A$ is isomorphic to one of the groups from~\cite[Theorem~3.3]{PV}. Assume that $A$ is noncyclic. An inspection of the list of possible abelian groups~$A$ from~\cite[Theorem~3.3]{PV} implies that $G$ has a subgroup isomorphic to one of the following groups:
$$E_9\rtimes C_2,~C_2\times D_8,~E_4\times D_{2p},$$
where $p$ is an odd prime. To obtain a contradiction, it suffices to prove that the above groups are not Schur. The first two of them and the third one if $p=3$ are not Schur by computer calculations~\cite{Ziv} (see also~\cite[Lemma~3.1]{PV}). Further, we are going to modify a construction of a nonschurian $S$-ring over $D_8\times C_p$, $p\geq 5$, from~\cite[p.~1081]{Ry3} to obtain a nonschurian $S$-ring over $D_8\times C_p$ which is isomorphic to an $S$-ring over $E_4\times D_{2p}$. 

Let $H=\langle a,b:~a^4=b^2=e,a^b=a^{-1} \rangle\cong D_8$, $C$ a cyclic group of order~$p$, and $c$ a generator of~$C$. Put
$$Z_0=\{e\},~Z_1=\{a^2\},~Z_2=\{ab\},~Z_3=\{a^3b\},~Z_4=\{a,a^3,b,a^2b\},$$
$$X_1=c\{e,a^2\},~X_2=c\{ab,a^3b\},~X_3=c\{a,b\},~X_4=c\{a^3,a^2b\},$$
$$Y_1=c^{-1}\{e,a^2\},~Y_2=c^{-1}\{ab,a^3b\},~Y_3=c^{-1}\{a^3,b\},~Y_4=c^{-1}\{a,a^2b\},$$
$$T_{1k}=c^k\{e,a^2\},~T_{2k}=c^k\{ab,a^3b\},~T_{3k}=c^k\{a,a^3,b,a^2b\},~k\in\{2,\ldots,p-2\}.$$
Due to~\cite[Lemma~3.2]{Ry3}, the above sets form a partition of $H\times C$ that defines an $S$-ring $\mathcal{A}$ over $H\times C$. From~\cite[Lemma~3.4]{Ry3} it follows that $\mathcal{A}$ is nonschurian. Let $\varphi \in \aut(H\times C)$ such that $a^\varphi=a^{-1}$, $b^\varphi=b$, and $c^\varphi=c^{-1}$. An explicit verifying yields that $\varphi$ interchanges the sets $Z_2$ and $Z_3$, $X_i$ and $Y_i$, where $i\in\{1,2,3,4\}$, $T_{lk}$ and $T_{l,p-k}$, where $l\in \{1,2,3\}$ and $k\in\{2,\ldots,p-2\}$.

Let $\mathcal{B}=\mathcal{A}^{\langle \varphi \rangle}$. Put $R=\langle a^2_r,(ab)_r,c_r,a_r\varphi\rangle$. By the definition of $\mathcal{B}$, we have $R\leq \aut(\mathcal{B})$. A straightforward computation implies that $|a_r\varphi|=2$, the elements $a^2_r$, $(ab)_r$, and $a_r\varphi$ pairwise commute, and $c_r^{a_r\varphi}=c_r^{-1}$. Therefore 
$$R=\langle a^2_r,(ab)_r \rangle \times (C_r\rtimes \langle a_r\varphi\rangle) \cong E_4 \times D_{2p}.$$
One can see that the orbits of $\langle a^2_r,(ab)_r \rangle \times C_r$ are exactly the right cosets by a group $\langle a^2,ab \rangle \times C$ of order~$4p$ and $e^{a_r\varphi}=a^{-1}$. So we conclude that $R$ is transitive and hence regular. Thus, $\mathcal{B}$ is isomorphic to an $S$-ring over~$E_4\times D_{2p}$.

It remains to verify that $\mathcal{B}$ is nonschurian. Assume the contrary. Then the basic sets of $\mathcal{B}$ are the orbits of $K=\aut(\mathcal{B})_e$. Let $f\in K_{ac}$. Then $((X_4\cup Y_4)ac)^f=(X_4\cup Y_4)ac$, $((T_{12}\cup T_{1,p-2})ac)^f=(T_{12}\cup T_{1,p-2})ac$, $(T_{12}\cup T_{1,p-2})^f=T_{12}\cup T_{1,p-2}$, and $(T_{33}\cup T_{3,p-3})^f=T_{33}\cup T_{3,p-3}$. So
$$\{c^2\}^f=((X_4\cup Y_4)ac\cap (T_{12}\cup T_{1,p-2}))^f=(X_4\cup Y_4)ac\cap (T_{12}\cup T_{1,p-2})=\{c^2\}$$
and
$$\{ac^3,a^3c^3\}^f=((T_{12}\cup T_{1,p-2})ac\cap (T_{33}\cup T_{3,p-3}))^f=(T_{12}\cup T_{1,p-2})ac\cap (T_{33}\cup T_{3,p-3})=\{ac^3,a^3c^3\}.$$
The first of the above equalities implies that $((X_3\cup Y_3)c^2)^f=(X_3\cup Y_3)c^2$ and using also the second one, we obtain
$$\{ac^3\}^f=((X_3\cup Y_3)c^2\cap \{ac^3,a^3c^3\})^f=(X_3\cup Y_3)c^2\cap \{ac^3,a^3c^3\}=\{ac^3\}.$$
Thus, $f\in K_{ac^3}$ and consequently 
\begin{equation}\label{stab}
K_{ac}\leq K_{ac^3}. 
\end{equation}

On the one hand, $|K_{ac}|=|K|/|X_3\cup Y_3|=|K|/4$ because $X_3\cup Y_3$ is an orbit of~$K$ containing~$ac$. On the other hand, $|K_{ac^3}|=|K|/|T_{33}\cup T_{3,p-3}|=|K|/8$ because $T_{33}\cup T_{3,p-3}$ is an orbit of~$K$ containing~$ac^3$. Therefore $|K_{ac}|=|K|/4>|K|/8=|K_{ac^3}|$, a contradiction to Eq.~\eqref{stab}.

\section{Nonschurity of groups of order~$2pn$}
 
One of the crucial steps towards the proof of Theorem~\ref{main2} is the theorem below that is also of an independent interest. 

\begin{theo}\label{pn}
Let $p$ be a prime and $n>2$ an integer such that $2n$ divides $p-1$. Then the groups $C_p\times D_{2n}$ and $D_{2pn}$ are not Schur.
\end{theo}

To prove Theorem~\ref{pn}, we need some preliminary work. Let $p$ and $n$ be as in Theorem~\ref{pn}, $H=\langle a,b:~a^n=b^2=e,~a^b=a^{-1}\rangle\cong D_{2n}$, $A=\langle a \rangle$, $B=\langle b \rangle$, $C$ a cyclic group of order~$p$, and $G=H\times C$.

The group $\aut(C)$ is a cyclic group of order~$p-1$. Let $\sigma$ be a generator of $\aut(C)$. Since $p-1$ is divisible by~$n$, the group $\aut(C)$ has a unique subgroup~$M$ of index~$n$. Put $m=|M|=\frac{p-1}{n}$. It is easy to see that $|\orb(M,C^\#)|=n$ and $\aut(C)$ acts on the set $\orb(M,C^\#)$ as a regular cyclic group of order~$n$. Let us choose $C_0\in \orb(M,C^\#)$ and put $C_i=C_0^{\sigma^i}$, $i\in \mathbb{Z}_n$. Clearly, $|C_0|=\ldots=|C_{n-1}|=m$ and $\sigma$ induces the permutation $(C_0\ldots C_{n-1})$ on the set $\orb(M,C^\#)$. Since $\sigma$ is a generator of $\aut(C)$, we conclude that $C_{i+1}=C_{i}^{\sigma}=C_i^{(r)}$, where $r$ is a primitive root modulo~$p$. Note that $|M|=\frac{p-1}{n}$ is even because $2n$ divides $p-1$. This implies that $M$ contains an automorphism $\sigma_0$ of order~$2$ which inverses every element of $C$. Therefore $C_i=C_i^{-1}$ for every $i\in \mathbb{Z}_n$.

One can see that $\orb(M,C)=\{\{e\},C_0,\ldots,C_{n-1}\}$ is the set of the basic sets of the cyclotomic $S$-ring $\mathcal{C}=\cyc(M,C)$. Given $i,j,k\in \mathbb{Z}_n$, put $c_{i,j}^k=c_{C_iC_j}^{C_k}$ and $c_{i^{\sigma^l}, j^{\sigma^l}}^{k^{\sigma^l}}=c_{C_i^{\sigma^l} C_j^{\sigma^l}}^{C_k^{\sigma^l}}$. From Lemma~\ref{burn} it follows that $\sigma$ and hence $\sigma^l$ are algebraic automorphisms of $\mathcal{C}$. So
\begin{equation}\label{constants1}
c_{i+l,j+l}^{k+l}=c_{i^{\sigma^l},~j^{\sigma^l}}^{k^{\sigma^l}}=c_{i,j}^{k}
\end{equation}
for all $i,j,k,l\in \mathbb{Z}_n$.

\begin{lemm}\label{cproducts}
Let $i\in \mathbb{Z}_n$. Then
$$\sum \limits_{j\in \mathbb{Z}_n} \underline{C_j}\cdot\underline{C_{j+i}}=\delta_{i0}(p-1)e+(m-\delta_{i0})\underline{C}^\#,$$
where $\delta_{i0}$ is the Kronecker delta.
\end{lemm}

\begin{proof}
It follows that
\begin{equation}\label{sum1}
\sum \limits_{j\in \mathbb{Z}_n} \underline{C_j}\cdot\underline{C_{j+i}}=\sum \limits_{j\in \mathbb{Z}_n} \left(\delta_{i0}me+\sum \limits_{k\in \mathbb{Z}_n} c_{j,j+i}^k \underline{C_k}\right)=\delta_{i0}(p-1)e+\sum \limits_{k\in \mathbb{Z}_n} \Lambda_k \underline{C_k},
\end{equation}
where
$$\Lambda_k=\sum \limits_{j\in \mathbb{Z}_n} c_{j,j+i}^k.$$
By Eq.~\eqref{constants1}, we have $c_{j,j+i}^k=c_{j+l-k,j+i+l-k}^l$ for all $k,l\in \mathbb{Z}_n$. Since $j$ runs over the whole set $\mathbb{Z}_n$, the latter implies that $\Lambda_k=\Lambda_l$ for all $k,l\in \mathbb{Z}_n$. Together with Eq.~\eqref{sum1}, this yields that 
$$\sum \limits_{j\in \mathbb{Z}_n} \underline{C_j}\cdot\underline{C_{j+i}}=\delta_{i0}(p-1)e+\Lambda_0\underline{C}^\#.$$
The total number of elements in the left-hand side of the above equality is $n|C_0|^2=(p-1)m$ and hence $\Lambda_0=m-\delta_{i0}$ as required. 
\end{proof}

Let us construct a nonschurian $S$-ring over~$G$. Put

$$X_0=\{e\},~X_1=A^\#,~X_2=Ab=H\setminus A,$$
$$Y_1=C^\#,~Y_2=A^\# C^\#,$$
$$Z_1=\bigcup \limits_{i\in \mathbb{Z}_n} a^ib C_i,~Z_2=\bigcup \limits_{i\in \mathbb{Z}_n} a^ib (C^\#\setminus C_i).$$

The sets $X_0,X_1,X_2,Y_1,Y_2,Z_1,Z_2$ form a partition of $G$. Denote this partition by $\mathcal{S}$.

\begin{lemm}\label{sringcpd2n}
The $\mathbb{Z}$-module $\mathcal{A}=\Span_{\mathbb{Z}}\{\underline{T}:~T\in \mathcal{S}\}$ is an $S$-ring over $G$.
\end{lemm}

\begin{proof}
Observe that the sets $X_0,X_1,X_2,Y_1,Y_2,Z_1\cup Z_2$ are exactly the basic sets of the $S$-ring $\mathcal{B}=(\mathcal{T}_{A}\wr \mathcal{T}_{G/A})\otimes \mathcal{T}_C$. This implies that
$$\mathcal{A}=\Span_{\mathbb{Z}}\{\underline{T},\underline{Z_1}:~T\in \mathcal{S}(\mathcal{B})\}.$$ 
One can see that $Z_1=Z_1^{-1}$ and $Z_2=Z_2^{-1}$ because $C_i=C_i^{-1}$ for every $i\in \mathbb{Z}_n$. So to prove the lemma, it remains to verify that $\underline{Z_1}\cdot \underline{T},\underline{T}\cdot \underline{Z_1}\in \mathcal{A}$ for every $T\in \mathcal{S}$. Since $Y_2=X_1Y_1$ and $Z_2$ is a complement to the union of all basic sets distinct from~$Z_2$, it suffices to check the above inclusions only for $T\in\{X_1,X_2,Y_1,Z_1\}$. It is easy and straightforward that
$$\underline{Z_1}\cdot \underline{X_1}=\underline{X_1}\cdot \underline{Z_1}=\underline{Z_2},$$
$$\underline{Z_1}\cdot \underline{X_2}=\underline{X_2}\cdot \underline{Z_1}=\underline{Y_1}+\underline{Y_2},$$
and
$$\underline{Z_1}\cdot \underline{Y_1}=\underline{Y_1}\cdot \underline{Z_1}=m(\underline{X_2}+\underline{Z_2})+(m-1)\underline{Z_1}.$$
Finally, let us compute $\underline{Z_4}^2$:
$$\underline{Z_4}^2=\sum \limits_{i\in \mathbb{Z}_n} a^i\left(\sum \limits_{j\in \mathbb{Z}_n} \underline{C_j}\cdot\underline{C_{j+i}}\right)=(p-1)e+(m-1)\underline{C}^\#+m\sum \limits_{i\in \mathbb{Z}_n\setminus \{0\}} a^i\underline{C}^\#=(p-1)e+(m-1)\underline{Y_1}+m\underline{Y_2},$$
where the second equality holds by Lemma~\ref{cproducts}.
\end{proof}

\begin{lemm}\label{nonschur}
The $S$-ring  $\mathcal{A}$ is not schurian.
\end{lemm}

\begin{proof}
Assume the contrary that $\mathcal{A}$ is schurian. Then $\mathcal{S}=\orb(K,G)$, where $K=\aut(\mathcal{A})_e$. Let $i,j,k\in \mathbb{Z}_n$ be pairwise distinct, $c_j\in C_j$, and $c_k\in C_k$. Since $Z_2$ is an orbit of $K$ and $a^ibc_j,a^ibc_k\in Z_2$, there is $f\in K$ such that $(a^ibc_j)^f=a^ibc_k$. The group $C$ is an $\mathcal{A}$-subgroup and hence $Ca^ib=a^ibC$ is a block of $K$. One can see that $a^ibc_k=(a^ibc_j)^f\in a^ibC \cap (a^ibC)^f$. So $(a^ibC)^f=a^ibC$. Together with $X_2^f=X_2$, this implies that 
$$\{a^ib\}^f=(a^ibC\cap X_2)^f=(a^ibC)^f\cap X_2^f=a^ibC\cap X_2=\{a^ib\}.$$
Thus, $(a^ib)^f=a^ib$.

Due to the equalities $(a^ibc_j)^f=a^ibc_k$ and $(a^ib)^f=a^ib$ from the previous paragraph and $Y_1^f=Y_1$, we obtain
$$(C_ic_j)^f=(Z_1a^ibc_j\cap Y_1)^f=(Z_1a^ibc_j)^f\cap Y_1^f=Z_1a^ibc_k\cap Y_1=C_ic_k$$
and 
$$(C_i)^f=(Z_1a^ib\cap Y_1)^f=(Z_1a^ib)^f\cap Y_1^f=Z_1a^ib\cap Y_1=C_i.$$
The above equalities yield that $|C_i\cap C_ic_j|=|C_i\cap C_ic_k|$. Together with $C_i=C_i^{-1}$, the latter implies that
\begin{equation}\label{strucconst}
c_{ii}^j=c_{ii}^k.
\end{equation}

From Eq.~\eqref{strucconst} it follows that 
$$\underline{C_i}\cdot\underline{C_i}=me+\lambda \underline{C_i}+\mu (\underline{C}^\#-\underline{C_i})$$
for some positive integers $\lambda$ and $\mu$. The latter means that the partition of $C$ into subsets $\{e\}$, $C_i$, $C^\#\setminus C_i$ defines an $S$-ring of rank~$3$ over $C$. Since $n>2$, we obtain a contradiction to Lemma~\ref{cyclprime}(2). 
\end{proof}

\begin{lemm}\label{iso}
The $S$-ring $\mathcal{A}$ is isomorphic to an $S$-ring over $D_{2pn}$.
\end{lemm}

\begin{proof}
To prove the lemma, it suffices to find a regular subgroup isomorphic to $D_{2pn}$ in $\aut(\mathcal{A})$. Let $f=b_r(\id_H\times \sigma_0)\in \Hol(G)$ and $R=\langle (ac)_r, f \rangle$. Since each $C_i$ is inverse-closed, every basic set of $\mathcal{A}$ is invariant under the action of $\id_H\times \sigma_0$ and hence $\id_H\times \sigma_0\in \aut(\mathcal{A})$. Thus, $R\leq \aut(\mathcal{A})$. 

It can be verified straightforwardly that the permutation of $G$ induced by $f$ is a product of $pn$ independent cycles of length~$2$ of the form  
\begin{equation}\label{orbits}
\prod \limits_{x\in A\times C} (x~x^{f}),
\end{equation}
where given $x=a^ic^j\in A\times C$,
$$x^{f}=a^ibc^{-j}.$$
So the group $\langle f \rangle$ is a semiregular group of order~$2$. One can compute that $f^{-1}(ac)_rf=(a^{-1}c^{-1})_r$. Therefore $R=\langle (ac)_r \rangle \rtimes \langle f \rangle\cong D_{2pn}$ and hence $|R|=2pn=|G|$. In view of Eq.~\eqref{orbits}, each orbit of $\langle f \rangle$ contains a unique element from $A\times C$ whereas $\langle (ac)_r \rangle$ acts transitively on $A\times C$. Together with the previous discussion, this yields that $R$ is transitive and hence regular on $G$. Thus, $R$ is a regular subgroup of $\aut(\mathcal{A})$ isomorphic to $D_{2pn}$.
\end{proof}

Theorem~\ref{pn} immediately follows from Lemma~\ref{nonschur} and Lemma~\ref{iso}.

\begin{lemm}\label{isofrob}
If $\GCD(\frac{p-1}{2n},n)=1$, then $\mathcal{A}$ is isomorphic to an $S$-ring $C_p\rtimes C_{2n}$. 
\end{lemm}

\begin{proof}
To prove the lemma, we are going to find a regular subgroup of $\aut(\mathcal{A})$ isomorphic to $C_p\rtimes C_{2n}$. Let $\alpha\in \aut(H)\cong \aut(D_{2n})$ such that $a^{\alpha}=a$ and $b^{\alpha}=ab$. Observe that $|\alpha\times\sigma|=p-1$ because $|\sigma|=p-1$ and $|\alpha|=n$ divides $p-1$. Put $\beta=\alpha_1\times \sigma_1$, where $\alpha_1=\alpha^{\frac{p-1}{2n}}$ and $\sigma_1=\sigma^{\frac{p-1}{2n}}$. The definition of $\beta$ is correct because $2n$ divides~$p-1$. Clearly, $|\beta|=2n$. 

Let $f=b_r\beta\in \Hol(G)$ and $R=\langle c_r, f \rangle$. Since $\sigma$ induces the permutation $(C_0\ldots C_{n-1})$ on the set $\orb(M,C^\#)$, the set $Z_1$ is invariant under the action of $\alpha \times \sigma$ and consequently every basic set of $\mathcal{A}$ so is. This yields that $\alpha \times \sigma$ is a Cayley automorphism of $\mathcal{A}$ and hence $\beta$ so is. Therefore $R\leq \aut(\mathcal{A})$.

Note that $|\alpha_1|=n$ because $\GCD(\frac{p-1}{2n},n)=1$. So $b^{\alpha_1}=a_1b$, where $a_1$ is a generator of $A$. An explicit computation using the equality $b^{\alpha_1^i}=a_1^ib$ implies that the permutation of $G$ induced by $f$ is a product of $p$ independent cycles of length~$2n$ of the form
\begin{equation}\label{orbits2}
\prod \limits_{x\in C} (x~x^{f}~\ldots~x^{f^{2n-1}}),
\end{equation}
where
$$x^{f^l}=x^{\beta^l}a_1^{[\frac{l+1}{2}]}b^l,~l\in\{0,\ldots,2n-1\}.$$
So the group $\langle f \rangle$ is a semiregular group of order~$2n$. It is easy to verify that $f^{-1}c_rf=c^{\sigma_1}_r$ and hence $R=C_r\rtimes \langle f\rangle\cong C_p\rtimes C_{2n}$. In particular, $|R|=|2pn|=|G|$. Due to Eq.~\eqref{orbits2}, each of $p$ orbits of $\langle f \rangle$ contains an element from $C$ whereas obviously $C_r$ is transitive on $C$. This implies that $R$ is transitive on $G$. Thus, $R$ is a regular subgroup of $\aut(\mathcal{A})$ isomorphic to $C_p\rtimes C_{2n}$.
\end{proof}

\begin{corl}\label{frobp}
In the conditions of Lemma~\ref{isofrob}, the Frobenius group $C_p\rtimes C_{2n}$ is not Schur for every $n>2$. In particular, the Frobenius group $C_p\rtimes C_{p-1}$ is not Schur for every odd prime~$p\geq 7$.
\end{corl}

\section{Order of a dihedral Schur group}

In this section, we provide several restrictions on the order of a dihedral Schur group and, in particular, prove Theorem~\ref{main2}. Throughout this section, $n\geq 3$ is a positive integer, $G=\langle a,b:~a^n=b^2=e,~a^b=a^{-1} \rangle \cong D_{2n}$, $A=\langle a \rangle \cong C_n$, and $B=\langle b \rangle \cong C_2$. 

\begin{lemm}\label{symmetric}
Every symmetric $S$-ring over a cyclic group of even order is isomorphic to an $S$-ring over a dihedral group of the same order.
\end{lemm}

\begin{proof}
Let $\mathcal{A}$ be a symmetric $S$-ring over a group $H\cong C_{2n}$, $H_0$ a subgroup of $H$ of index~$2$, $x$ a generator of $H_0$, $y\in H\setminus H_0$, and $\sigma_0\in \aut(H)$ such that $h^{\sigma_0}=h^{-1}$, $h\in H$. Put $R=\langle x_r,y_r\sigma_0 \rangle$. Since $\mathcal{A}$ is symmetric, $\sigma_0\in \aut(\mathcal{A})$ and hence $R\leq \aut(\mathcal{A})$. To prove the lemma, it suffices to verify that $R\cong D_{2n}$ and $R$ is regular on $H$. A straightforward computation implies that $|y_r\sigma_0|=2$ and $(x_r)^{y_r\sigma_0}=x^{-1}_r$. Therefore $R=H_0 \rtimes \langle y_r\sigma_0\rangle \cong D_{2n}$. Clearly, the orbits of $(H_0)_r$ are $H_0$ and $H\setminus H_0$. One can see that $e^{y_r\sigma_0}=y^{-1}\in H\setminus H_0$. This yields that $R$ is transitive on $H$ and consequently regular as required. 
\end{proof}

\begin{corl}\label{dihedrestr}
If $G$ is Schur, then $2n$ is one of the forms $p^k$, $pq^k$, $2pq^k$, $pqr$, $2pqr$, where $p$, $q$, and $r$ are primes and $k\geq 1$.
\end{corl}

\begin{proof}
Assume the contrary. Then there exists a nonschurian symmetric $S$-ring $\mathcal{A}$ over a cyclic group of order~$2n$ by~\cite[Theorem~2.1]{EKP1}. Lemma~\ref{symmetric} implies that $\mathcal{A}$ is isomorphic to an $S$-ring over $G$ and hence $G$ is not Schur.
\end{proof}

The statement below follows from~\cite[Corollary~5.3]{PV} and Lemma~\ref{schursection}.

\begin{lemm}\label{3mod}
Suppose that $G$ is Schur and $p$ is a prime divisor of $n$ with $p \equiv 3 \bmod 4$. Then $p\in\{3,7,11\}$.
\end{lemm}

The next lemma is a corrected version of~\cite[Corollary~5.4]{PV}, where the case of a Paley difference set with parameters~$(11,5,2)$ is missed. In fact,~\cite[Corollary~5.4]{PV} states that a difference set in a cyclic subgroup of index~$2$ of a dihedral Schur group is isomorphic (up to taking a complement) to a difference set from~\cite[Statement~35]{Demb} which is a Singer difference set or a Paley difference set with parameters~$(11,5,2)$.

\begin{lemm}\label{schurdiff}
Suppose that $G$ is Schur and $A$ has a nontrivial difference set~$D$. Then $D$ is isomorphic (up to taking a complement) to a Paley difference set with parameters~$(11,5,2)$ or a Singer difference set with parameters~$(\frac{q^{d+1}-1}{q-1},\frac{q^{d}-1}{q-1},\frac{q^{d-1}-1}{q-1})$, where $q$ is a prime power and $d\geq 2$ is an integer. 
\end{lemm}

\begin{prop}\label{pdiff}
Suppose that $p$ is a prime divisor of~$n$ such that $C_p$ has a nontrivial difference set and $\frac{n}{p}\geq 3$. Then $G$ is not Schur.
\end{prop}

\begin{proof}
Clearly, $p$ is odd. In view of Lemma~\ref{schursection}, we may assume that $n=pq$, where $q$ is an odd prime, or $n=4p$. Let $P$ and $Q$ be  subgroups of~$A$ of order~$p$ and index~$p$, respectively, and $M\leq\aut(A)$ defined as follows. If $q=p$, then $M$ is a subgroup of $\aut(A)$ of order~$p-1$; otherwise 
$$M=\{(\sigma,\tau)\in \aut(P)\times \aut(Q):~(\sigma^{\pi_1})^\psi=\tau^{\pi_2}\},$$
where $\pi_1$ and $\pi_2$ are the canonical epimorphisms from $\aut(P)$ and $\aut(Q)$, respectively, to their quotients by the subgroups of index~$2$, and $\psi$ is a unique isomorphism between these quotients. Put $\mathcal{A}_1=\cyc(M,A)$. The definition of $\mathcal{A}_1$ implies that $(\mathcal{A}_1)_{S}=\mathcal{T}_{S}$, where $S=A/Q\cong C_p$. One can see that $\mathcal{A}_1$ is not a nontrivial generalized wreath or tensor product and hence $\mathcal{A}_1$ is normal by~\cite[Theorem~4.4.1]{CP}. 

Let $D$ be a difference set in $S\cong C_p$ and $\mathcal{A}_2=\mathcal{A}(D)$ (see Section~$3$ for the definition) an $S$-ring over $G/Q\cong D_{2p}$. Clearly, $(\mathcal{A}_2)_{S}=\mathcal{T}_{S}$. Put
$$\mathcal{A}=\mathcal{A}_1 \wr_S \mathcal{A}_2.$$
Note that $\mathcal{A}$ is well-defined because $(\mathcal{A}_1)_{S}=(\mathcal{A}_2)_{S}=\mathcal{T}_{S}$. To prove the proposition, it suffices to verify that $\mathcal{A}$ is nonschurian. Assume the contrary. According to~\cite[Corollary~5.7]{EP2}, $\mathcal{A}$ is schurian if and only if so are $\mathcal{A}_1$ and $\mathcal{A}_2$ and there exist groups $K_1\leq \aut(\mathcal{A}_1)$ and $K_2\leq \aut(\mathcal{A}_2)$ such that $K_1\geq A_r$, $K_2\geq (G/Q)_r$, $\mathcal{A}_1=V(K_1,A)$, $\mathcal{A}_2=V(K_2,G/Q)$, and
\begin{equation}\label{autrest}
K_1^S=K_2^S.
\end{equation} 
Clearly, in this case
\begin{equation}\label{autsect}
\mathcal{A}_S=V(K_1^S,S)=V(K_2^S,S).
\end{equation}
Since $\mathcal{A}_1$ is normal, $K_1\leq\aut(\mathcal{A}_1)\leq \Hol(G)$ by~\cite[Theorem~4.5]{EP1} and hence $K_1^S\leq \Hol(S)$. Due to Eq.~\eqref{autsect}, the set $S^\#$ is an orbit of the stabilizer of the identity element of $S$ in $K_1^S$. This yields that $|K_1^S|\geq |S|(|S|-1)=|\Hol(S)|$. Therefore 
\begin{equation}\label{authol}
K_1^S=\Hol(S). 
\end{equation}
Eqs.~\eqref{autrest} and~\eqref{authol} imply that $K_2^S=\Hol(S)$. In particular, there is $f\in K_2$ such that $S^f=S$ and $f^S$ inverses every element of~$S$. From Lemma~\ref{autmult} it follows that $f^S$ is a multiplier of $D$, a contradiction to Lemma~\ref{minus1}.
\end{proof}

\begin{corl}\label{pdiffcorl}
Suppose that $G$ is Schur and $p$ is a prime divisor of~$n$ such that $C_p$ has a nontrivial difference set~$D$. Then $n\in\{p,2p\}$ and $D$ is isomorphic (up to taking a complement) to a Paley difference set with parameters~$(11,5,2)$ or a Singer difference set with parameters~$(q^2+q+1,q+1,1)$, where $q$ is a prime power.
\end{corl}

\begin{proof}
Proposition~\ref{pdiff} implies that $n\in\{p,2p\}$, whereas Lemma~\ref{schurdiff} implies that $D$ is isomorphic (up to taking a complement) to a Paley difference set with parameters~$(11,5,2)$ or a Singer difference set with parameters~$(p,k,\lambda)=(\frac{q^{d+1}-1}{q-1},\frac{q^{d}-1}{q-1},\frac{q^{d-1}-1}{q-1})$, where $q$ is a prime power and $d\geq 2$. It remains to verify that $d=2$ in the latter case. If $d\geq 4$, then an explicit computation yields that
$$1+\sqrt{k}=1+\sqrt{\frac{q^{d+1}-1}{q-1}}>\frac{v-1}{k}=q$$
and hence the design corresponding to~$D$ is not $2$-transitive by~\cite[Theorem~8.3]{Kantor}. So the $S$-ring $\mathcal{A}(D)$ over $D_{2p}$ is nonschurian by~\cite[Lemma~5.2]{PV}. Therefore $D_{2p}$ and consequently $G$ are not Schur, a contradiction to the assumption of the corollary. Thus, $d\in \{2,3\}$. However, if $d=3$, then $p=(q^2+1)(q+1)$ is not prime, a contradiction. 
\end{proof}

\begin{prop}\label{d8}
Suppose that $G$ is Schur and $n$ is divisible by~$4$. Then $n$ is a power of~$2$ or $n=12$. 
\end{prop}

\begin{proof}
We are done if $n$ is a power of~$2$. Suppose that $n=2^i3^j$ for some $i\geq 2$ and $j\geq 1$. If $i\geq 3$ or $j\geq 2$, then $G$ has a subgroup isomorphic to $D_{48}$ or $D_{72}$, respectively. The group $D_{48}$ is not Schur by computer calculations~\cite{Ziv}, whereas the group $D_{72}$ is not Schur by Corollary~\ref{dihedrestr}. In both cases, we obtain a contradiction to Lemma~\ref{schursection}. Therefore $i=2$ and $j=1$ and hence $n=12$.

In view of the above paragraph, we may assume that $n$ has an odd prime divisor $p\geq 5$. Due to Lemma~\ref{schursection}, it suffices to verify that the group $D_{8p}$ is not Schur. If $p\equiv 3 \bmod 4$, then $p\in\{7,11\}$ by Lemma~\ref{3mod}. Lemma~\ref{diffcycl} implies that $C_p$ has a nontrivial difference set and $G$ is not Schur by Proposition~\ref{pdiff}. If $p-1$ is divisible by~$8$, then $D_{8p}$ is not Schur by Theorem~\ref{pn}. 

It remains to consider only the case when $p-1$ is divisible by~$4$ and not divisible by~$8$. In this case, $\frac{p-1}{4}$ is odd. We are going to prove that in this case a nonschurian $S$-ring from~\cite[p.~1083]{Ry3} over the group $Q_8 \times C_p$ is isomorphic to an $S$-ring over $D_{8p}$. Recall the construction. Let $H=\langle x,y:~x^4=e,x^2=y^2,x^y=x^{-1} \rangle \cong Q_8$, $C\cong C_p$, and $C_0$, $C_1$, $C_2$, $C_3$ the orbits of the subgroup $M$ of $\aut(C)\cong C_{p-1}$ of index~$4$ such that the cycle $(C_0 C_1 C_2 C_3)\in \sym(\orb(M,C^\#))$ is a generator of $\aut(C)$ acting on $\orb(M,C^\#)$. Since $\frac{p-1}{4}$ is odd, $C_0^{-1}=C_2$ and $C_1^{-1}=C_3$. Put
$$Z_0=\{e\},~Z_1=\{x^2\},~Z_2=\{xy,x^3y\},~Z_3=\{x,x^3,y,x^2y\},$$
$$Z_4=C^{\#},~Z_5=x^2C^{\#}$$
$$Z_6=\{x,x^3\}(C_1\cup C_3)\cup \{y,x^2y\}(C_0\cup C_2),$$
$$Z_7=xC_0\cup x^3C_2\cup yC_1\cup x^2yC_3,$$
$$Z_8=Z_7x^2=xC_2\cup x^3C_0\cup  yC_3\cup x^2yC_1,$$
$$Z_9=\{xy,x^3y\}C^\#.$$
Due to~\cite[Lemma~4.4]{Ry3}, the module $\mathcal{A}=\Span_{\mathbb{Z}}\{\underline{Z_i}:~i\in \{0,\ldots,9\}\}$ is an $S$-ring over $H\times C\cong Q_8 \times C_p$. From~\cite[Lemma~4.5]{Ry3} it follows that $\mathcal{A}$ is nonschurian. Let $\sigma\in \aut(H\times C)$ such that $x^\sigma=x^{-1}$, $y^\sigma=y^{-1}=x^2y$, and $\sigma^C$ inverses every element of $C$. It is easy to see that the sets $Z_i$, $i\in \{0,\ldots,9\}$, are invariant under action of $\sigma$ and hence $\sigma\in \aut(\mathcal{A})$. Put $R=\langle (xyc)_r,y_r\sigma \rangle$. A straightforward computation implies that $|y_r\sigma|=2$ and $(xyc)_r^{y_r\sigma}=(xyc)^{-1}_r$. Therefore 
$$R=\langle (xyc)_r \rangle \rtimes \langle y_r\sigma\rangle\cong D_{8p}.$$
Clearly, the orbits of $\langle (xyc)_r \rangle$ are $\langle xyc \rangle$ and $\langle xyc \rangle y$. Since $e^{y_r\sigma}=y^{-1}\in \langle xyc \rangle y$, we conclude that $R$ is transitive on $H\times C$ and consequently regular. Thus, $\mathcal{A}$ is isomorphic to an $S$-ring over $D_{8p}$ and hence $D_{8p}$ is not Schur as required.
\end{proof}

Theorem~\ref{main2} immediately follows from Corollary~\ref{dihedrestr} and Proposition~\ref{d8}.

For a convenience, we provide below a statement collecting arithmetic restrictions on the order of a dihedral Schur group that follow from the previous results of the paper and are not covered by Theorem~\ref{main2}.

\begin{corl}\label{conditions}
Let a dihedral group of order~$2n$, where $n\geq 3$, be Schur.
\begin{enumerate}

\tm{1} If $p\in \pi(n)$ such that $p \equiv 3~\bmod~4$, then $p=3$ or $n\in\{7,11,14,22\}$.

\tm{2} If $p,q\in \pi(n)$ are odd and $p<q$, then $q~\not\equiv 1\bmod~p$.

\end{enumerate}

\end{corl}

\begin{proof}
Statement~$(2)$ immediately follows from Lemma~\ref{schursection} and Theorem~\ref{pn}. Let us prove Statement~$(1)$. From Lemma~\ref{3mod} it follows that $p\in\{3,7,11\}$. We are done if $p=3$. Lemma~\ref{diffcycl} implies that $C_p$ has a nontrivial difference set if $p=7$ or $p=11$. So $\frac{n}{p}\leq 2$ by Proposition~\ref{pdiff} and hence $n\in\{7,11,14,22\}$.
\end{proof}

\section{Schurity of $D_{2p}$}

In this section, we keep the notations from the previous one. In addition, we assume that $n=p$ is an odd prime and hence $G\cong D_{2p}$ and $A\cong C_p$. We are going to classify $S$-rings over~$D_{2p}$ for some~$p$ and use this classification to prove Theorem~\ref{main3}.

\begin{theo}\label{classification}
Let $p$ be a prime. Suppose that $p$ is a Fermat prime or $p=rq+1$, where $q$ is a prime and $r\in\{2,4\}$. Then for every $S$-ring $\mathcal{A}$ over $D_{2p}$, one of the following statements holds: 
\begin{enumerate}

\tm{1} $\rk(\mathcal{A})=2$;

\tm{2} $\mathcal{A}$ is cyclotomic;

\tm{3} $\mathcal{A}$ is isomorphic to an $S$-ring over~$C_{2p}$;

\tm{4} $p\in \{13,2q+1\}$ and $\mathcal{A}=\mathcal{A}(D)$ for some nontrivial difference set $D$ in $A$.

\end{enumerate}

\end{theo}

It should be mentioned that every $S$-ring over $C_{2p}$ is of rank~$2$, or cyclotomic, or a wreath product of two $S$-rings by~\cite[Theorem~4.1-4.2]{EP2} (see also~\cite{LM1,LM2}). The proof of Theorem~\ref{classification} will be given in the end of the section. We say that a subset $X$ of $G$ is \emph{mixed} if $X\cap A\neq \varnothing$ and $X\cap Ab\neq \varnothing$; otherwise we say that $X$ is \emph{nonmixed}. Throughout this section, $\mathcal{A}$ is an $S$-ring over $G$. The next lemma immediately follows from the fact that $G$ is a strong B-group (see~\cite[Theorem~26.6(i)]{Wi}).

\begin{lemm}\label{l0}
If $\mathcal{A}$ is primitive, then $\rk(\mathcal{A})=2$.  
\end{lemm}

Further we assume that $\mathcal{A}$ is imprimitive.

\begin{lemm}\label{l1}
If $G$ has an $\mathcal{A}$-subgroup of order~$2$, then $\mathcal{A}$ is cyclotomic or isomorphic to an $S$-ring over $C_{2p}$.
\end{lemm}

\begin{proof}
Let $L=\{e,s\}$ be an $\mathcal{A}$-subgroup. Then $\{s\}\in \mathcal{S}(\mathcal{A})$. Suppose that $A$ is an $\mathcal{A}$-subgroup. Lemma~\ref{cyclprime}(1) implies that $\mathcal{A}_A=\cyc(K_0,A)$ for some $K_0\leq \aut(A)$. Every basic set of $\mathcal{A}$ outside $A\cup L$ is of the form $Xs$, where $X\in \mathcal{S}(\mathcal{A}_A)=\orb(K_0,A)$. Therefore $\mathcal{A}=\cyc(K,G)$, where $K\leq \aut(G)$ is such that the restriction of $K$ on $A$ is equal to~$K_0$ and $K$ fixes $s$. Further we assume that $A$ is not an $\mathcal{A}$-subgroup.

Assume that there exists a nonmixed basic set $Y$ of $\mathcal{A}$ outside $L$. If $Y\subseteq A$, then $A=\langle Y\rangle$ is an $\mathcal{A}$-subgroup, a contradiction to the assumption. If $Y\subseteq Ab$, then $Ys$ is a basic set inside $A$ and $A=\langle Ys\rangle$ is an $\mathcal{A}$-subgroup, a contradiction to the assumption. 

Due to the above paragraph, we may assume that every basic set  $Y$ of $\mathcal{A}$ outside $L$ is mixed and hence of the form $Y=Y_0\cup Y_1s$, where $Y_0$ and $Y_1$ are nonempty subsets of $A$. Note that $Y_1s=(Y_1s)^{-1}$. So $Y=Y^{-1}$ and $Y_0=Y_0^{-1}$. Since $\{s\}\in \mathcal{S}(\mathcal{A})$, the set $Ys=Y_0s\cup Y_1$ is basic. Moreover, $Ys=(Ys)^{-1}$ and hence $Y_1=Y_1^{-1}$. Therefore $\sigma\in \aut(G)$ such that $a^\sigma=a^{-1}$ and $s^\sigma=s$ is a Cayley automorphism of $\mathcal{A}$. One can verify straightforwardly that $|\sigma s_r|=2$ and $a_r\sigma s_r=\sigma s_r a_r$. This implies that the group $R=\langle a_r, \sigma s_r \rangle$ is isomorphic to $C_{2p}$. It is easy to see that $R$ is transitive on $G$ and hence regular. Thus, $\aut(\mathcal{A})$ has a regular subgroup isomorphic to $C_{2p}$. This yields that $\mathcal{A}$ is isomorphic to an $S$-ring over $C_{2p}$. 
\end{proof}

Further we assume that 
$$A~\text{is a unique}~\mathcal{A}\text{-subgroup of}~G$$ 
In this case, $Ab$ is an $\mathcal{A}$-set. Due to Lemma~\ref{cyclprime}(1), we have $\mathcal{A}_A=\cyc(K,A)$ for some $K\leq \aut(A)\cong C_{p-1}$. Let $|K|=m$. Clearly, $m$ divides $p-1$ and every nontrivial basic set of $\mathcal{A}_A$ is of size~$m$. 

\begin{lemm}\label{easy}
If $Y\in \mathcal{S}(\mathcal{A})_{Ab}$, then $|Y|\notin \{1,p-1\}$.
\end{lemm}

\begin{proof}
Obviously, $|Y|\leq |Ab|=p$. If $|Y|=1$ ($|Y|=p-1$, respectively), then $\langle Y \rangle$ ($\langle Ab\setminus Y \rangle$, respectively) is an $\mathcal{A}$-subgroup of order~$2$, a contradiction to the assumption that $A$ is a unique $\mathcal{A}$-subgroup.
\end{proof}

\begin{lemm}\label{l2}
Let $r$ be a prime and $s$ a positive integer. Suppose that $r^s$ divides $m$. Then there exists a unique $Y\in \mathcal{S}(\mathcal{A})_{Ab}$ such that $r^s$ does not divide $|Y|$. Moreover, $|Y|\equiv 1~\bmod r^s$. 
\end{lemm}

\begin{proof}
Let $Y_1,Y_2\in \mathcal{S}(\mathcal{A})_{Ab}$ such that $Y_1\neq Y_2$. Clearly, $Y_1=Y_1^{-1}$ and $Y_2=Y_2^{-1}$. One can see that
$$\underline{Y_1}\cdot\underline{Y_1}=|Y_1|e+\sum_{\substack{X\in \mathcal{S}(\mathcal{A})_A \\ X\neq\{e\}}} c_{Y_1Y_1}^X \underline{X},$$
and 
$$\underline{Y_1}\cdot\underline{Y_2}=\sum_{\substack{X\in \mathcal{S}(\mathcal{A})_A \\ X\neq\{e\}}} c_{Y_1Y_2}^X \underline{X}.$$
Counting the number of elements in the left-hand side and the right-hand side of the above equalities and taking into account that $|X|=m$ for every nontrivial $X\in \mathcal{S}(\mathcal{A}_A)$, we obtain
\begin{equation}\label{divide}
m~|~|Y_1|(|Y_1|-1)~\text{and}~m~|~|Y_1||Y_2|.
\end{equation}

Observe that $\sum \limits_{Y\in \mathcal{S}(\mathcal{A})_{Ab}} |Y|=|Ab|=p$. This implies that there exists $Y\in \mathcal{S}(\mathcal{A})_{Ab}$ such that $r$ does not divide $|Y|$; for otherwise $r$ divides $p$ which is impossible. The second part of Eq.~\eqref{divide} yields that $|Y||Z|$ is divisible by $r^s$ and hence $|Z|$ is divisible by $r^s$ for every $Z\in \mathcal{S}(\mathcal{A})_{Ab}$ with $Z\neq Y$. From the first part of Eq.~\eqref{divide} it follows that $|Y|-1$ is divisible by $r^s$ or, equivalently, $|Y|\equiv 1~\bmod r^s$.
\end{proof}

\begin{lemm}\label{l3}
If $\GCD(|Y|,m)=1$ for some $Y\in \mathcal{S}(\mathcal{A})_{Ab}$, then $\mathcal{A}=\mathcal{A}_A\wr \mathcal{A}_{G/A}$.
\end{lemm}

\begin{proof}
From Lemma~\ref{l2} it follows that $|Z|$ is divisible by~$m$ for every $Z\in \mathcal{S}(\mathcal{A})_{Ab}$ with $Z\neq Y$. Note that $|Y|\geq 2$ by Lemma~\ref{easy}. Let $y_1,y_2\in Yb$ such that $y_1\neq y_2$ and $X$ a nontrivial basic set of $\mathcal{A}_A$ such that $y_2y_1^{-1}\in X$. Since $X\subseteq A$ and $Y\subseteq Ab$, we have
\begin{equation}\label{combination}
\underline{X}\cdot\underline{Y}=c_{XY}^{Y}\underline{Y}+\sum_{\substack{Z\in \mathcal{S}(\mathcal{A})_{Ab},\\Z\neq Y}}c_{XY}^Z\underline{Z}.
\end{equation}
Counting the number of elements in the left-hand side and the right-hand side of the above equality and taking into account that $|X|=m$ and $|Z|$ is divisible by~$m$ for every $Z\in \mathcal{S}(\mathcal{A})_{Ab}$ distinct from $Y$, we conclude that
\begin{equation}\label{divide2}
m|Y|=c_{XY}^Y|Y|+\alpha m
\end{equation}
for some nonnegative integer $\alpha$. Since $|Y|$ and $m$ are coprime, the number $c_{XY}^Y$ is divisible by~$m$ due to Eq.~\eqref{divide2}. Observe that $c_{XY}^Y\geq 1$ because $y_2y_1^{-1}\in X$ and the element $y_2b$ enters the element $\underline{X}\cdot\underline{Y}$ with nonzero coefficient. So $c_{XY}^Y\geq m$. Now Eq.~\eqref{divide2} yields that $c_{XY}^Y=m$ and $\alpha=0$. Therefore from Eq.~\eqref{combination} it follows that $\underline{X}\cdot\underline{Y}=|X|\underline{Y}$ and hence $X\leq \rad(Y)$. Thus, $A=\langle X \rangle\leq \rad(Y)$ and $Y=Ab$. The latter means that $\mathcal{A}=\mathcal{A}_A\wr \mathcal{A}_{G/A}$.
\end{proof}

\begin{lemm}\label{primepower}
If $m$ is a prime power, then $\mathcal{A}=\mathcal{A}_A\wr \mathcal{A}_{G/A}$. 
\end{lemm}

\begin{proof}
If $m=1$, then we are done by Lemma~\ref{l3}. If $m>1$, then there exists a unique $Y\in \mathcal{S}(\mathcal{A})_{Ab}$ such that $\GCD(|Y|,m)=1$ by Lemma~\ref{l2}. Again, we are done by Lemma~\ref{l3}.
\end{proof}

\begin{lemm}\label{4q1}
If $p=rq+1\geq 29$, where $q$ is a prime and $r\in\{2,4\}$, and $m\neq p-1$, then $\mathcal{A}=\mathcal{A}_A\wr \mathcal{A}_{G/A}$.
\end{lemm}

\begin{proof}
The condition of the lemma implies that $q$ is odd. Note that $m$ divides $p-1$. If $m\in\{1,2,4,q\}$, then we are done by Lemma~\ref{primepower}. Since $m\neq p-1$, it remains to consider the case $r=4$ and $m=2q$. In this case, $\rk(\mathcal{A}_A)=3$. Let $X_1$ and $X_2$ be nontrivial basic sets of $\mathcal{A}_A$. Clearly, $|X_1|=|X_2|=2q$. Since $m$ is even, $K$ contains an automorphism of order~$2$ and consequently $X_1=X_1^{-1}$ and $X_2=X_2^{-1}$. Thus, $\mathcal{A}$ is symmetric and hence commutative.

Lemma~\ref{l2} implies that there exist a unique $Y_1\in \mathcal{S}(\mathcal{A})_{Ab}$ such that $|Y_1|$ is not divisible by~$q$ and a unique $Y_2\in \mathcal{S}(\mathcal{A})_{Ab}$ such that $|Y_2|$ is not divisible by~$2$. If $Y_1=Y_2$, then $\GCD(|Y_1|,m)=1$ and we are done by Lemma~\ref{l3}.

Suppose that $Y_1\neq Y_2$. From Lemma~\ref{l2} it follows that $|Y_1|\equiv 1\bmod q$. Since $Y_2$ is a unique basic set from $\mathcal{S}(\mathcal{A})_{Ab}$ of odd cardinality and $Y_1\neq Y_2$, we conclude that $|Y_1|$ is divisible by~$2$. Therefore
$$|Y_1|\in\{q+1,3q+1\}.$$

Note that
\begin{equation}\label{zero}
c_{Y_1Z}^{T}=0
\end{equation}
whenever $Z,T\in \mathcal{S}(\mathcal{A}_A)$ or $Z,T\in \mathcal{S}(\mathcal{A})_{Ab}$. So by Eq.~\eqref{sum10}, we have
$$|Y_1|=c_{Y_1\{e\}}^{Y_1}+c_{Y_1X_1}^{Y_1}+c_{Y_1X_2}^{Y_1}=1+c_{Y_1X_1}^{Y_1}+c_{Y_1X_2}^{Y_1}.$$
Therefore 
\begin{equation}\label{y1sum}
c_{Y_1X_1}^{Y_1}+c_{Y_1X_2}^{Y_1}=|Y_1|-1.
\end{equation}
Eq.~\eqref{triangle} and symmetry of $\mathcal{A}$ imply that
$$c_{Y_1X_1}^{Y_1}|Y_1|=c_{Y_1Y_1}^{X_1}|X_1|~\text{and}~c_{Y_1X_2}^{Y_1}|Y_1|=c_{Y_1Y_1}^{X_2}|X_2|.$$
Since $|Y_1|$ is not divisible by~$q$ and $|X_1|=|X_2|=2q$, the above equalities yield that $c_{Y_1X_1}^{Y_1}$ and $c_{Y_1X_2}^{Y_1}$ are divisible by~$q$. If $|Y_1|=3q+1$, then $c_{Y_1X_1}^{Y_1}\geq 2q$ or $c_{Y_1X_2}^{Y_1}\geq 2q$ by Eq.~\eqref{y1sum}. The latter means that $Y_1X_1=Y_1$ or $Y_1X_2=Y_1$. In both cases, $\rad(Y_1)$ is nontrivial and hence $Y_1=Ab$, a contradiction to $Y_1\neq Y_2$.

Now suppose that $|Y_1|=q+1$. Then Eq.~\eqref{y1sum} implies that $c_{Y_1X_1}^{Y_1}=0$ or $c_{Y_1X_2}^{Y_1}=0$. Without loss of generality, let $c_{Y_1X_1}^{Y_1}=c_{X_1Y_1}^{Y_1}=0$. Put $Z_1=Y_1b\subseteq A$. Since $c_{X_1Y_1}^{Y_1}=0$, there is no an element from $Z_1$ which enters the element $\underline{X_1}\cdot \underline{Z_1}$. Therefore 
\begin{equation}\label{inter1}
|Z_1\cap X_1z|=0 
\end{equation}
for every $z\in Z_1$. The Cayley graph $\cay(A,X_1)$ is a Paley graph with parameters $(4q+1,2q,q-1,q)$. Due to Eq.~\eqref{inter1}, any two elements from $Z_1$ are non-adjacent in $\cay(A,X_1)$. This yields that  
\begin{equation}\label{inter2}
|X_1z\cap X_1z^\prime|=q
\end{equation}
for all $z,z^\prime\in Z_1$. 
Clearly, $|Z_1|=|Y_1|=q+1\geq 8$. Let $z,z^\prime,z^{\prime\prime}\in Z_1$ be pairwise distinct. Since $p\geq 29$,  
\begin{equation}\label{inter3}
|X_1z\cap X_1z^\prime\cap X_1z^{\prime\prime}|\geq 1.
\end{equation}
by~\cite[Corollary~0.2]{BM} stating that in this case any three distinct vertices in a Paley graph have a common neighbor. Using Eqs.~\eqref{inter1},~\eqref{inter2},~\eqref{inter3}, we obtain
$$|Z_1\cup X_1z\cup X_1z^\prime\cup X_1z^{\prime\prime}|=|Z_1|+|X_1z|+|X_1z^\prime|+|X_1z^{\prime\prime}|-|Z_1\cap X_1z|-|Z_1\cap X_1z^\prime|-|Z_1\cap X_1z^{\prime\prime}|-$$
$$-|X_1z\cap X_1z^\prime|-|X_1z\cap X_1z^\prime|-|X_1z^\prime\cap X_1z^{\prime\prime}|+|Z_1\cap X_1z\cap X_1z^\prime|+|Z_1\cap X_1z\cap X_1z^{\prime\prime}|+$$
$$+|Z_1\cap X_1z^\prime\cap X_1z^{\prime\prime}|+|X_1z\cap X_1z^\prime\cap X_1z^{\prime\prime}|-|Z_1\cap X_1z\cap X_1z^\prime\cap X_1z^{\prime\prime}|\geq$$
$$\geq q+1+6q-3q+1=4q+2>|A|,$$
a contradiction to $Z_1\cup X_1z\cup X_1z^\prime\cup X_1z^{\prime\prime}\subseteq A$.
\end{proof}

\begin{lemm}\label{wreathiso}
If $\mathcal{A}=\mathcal{A}_A\wr \mathcal{A}_{G/A}$, then $\mathcal{A}$ is isomorphic to an $S$-ring over $C_{2p}$.
\end{lemm}

\begin{proof}
Let $\sigma\in \aut(G)$ such that $a^\sigma=a$ and $b^\sigma=ab$. It is easy to check that the group $R=\langle b_r\sigma \rangle$ is a regular subgroup of $\aut(\mathcal{A})$ isomorphic to $C_{2p}$.
\end{proof}

\begin{proof}[Proof of Theorem~\ref{classification}]
The theorem for $p\in\{5,7,11,13,17,23\}$ follows from computer calculations (see~\cite{Ziv}). Further, we assume that $p\geq 29$.

If $\mathcal{A}$ is primitive, then $\rk(\mathcal{A})=2$ by Lemma~\ref{l0} and Statement~$(1)$ of Theorem~\ref{classification} holds. Suppose that $\mathcal{A}$ is imprimitive. If there is an $\mathcal{A}$-subgroup $L$ of order~$2$, then $\mathcal{A}$ is cyclotomic or $\mathcal{A}$ is isomorphic to an $S$-ring over $C_{2p}$ by Lemma~\ref{l1} and Statement~$(2)$ or~$(3)$ of Theorem~\ref{classification} holds. 

Now suppose that there is no an $\mathcal{A}$-subgroup of order~$2$. Since $\mathcal{A}$ is imprimitive, $A$ is a unique $\mathcal{A}$-subgroup. In view of Lemma~\ref{wreathiso}, to prove that Statement~$(3)$ of Theorem~\ref{classification} holds for $\mathcal{A}$, it suffices to verify that $\mathcal{A}=\mathcal{A}_A\wr \mathcal{A}_{G/A}$. If $p$ is a Fermat prime, then $m$ is a $2$-power and $\mathcal{A}=\mathcal{A}_A\wr \mathcal{A}_{G/A}$ by Lemma~\ref{primepower}. Let $p=rq+1$, where $q$ is a prime and $r\in\{2,4\}$. We may assume that $q$ is odd because otherwise $p$ is a Fermat prime and we are done. If $m\neq p-1$, then $\mathcal{A}=\mathcal{A}_A\wr \mathcal{A}_{G/A}$ by Lemma~\ref{4q1}. 

Let $m=p-1$. Then $\rk(\mathcal{A}_A)=2$ and $X=A^\#$ is a basic set of $\mathcal{A}_A$. Let $Y\in \mathcal{S}(\mathcal{A})_{Ab}$. If $|Y|=p$, then $Y=Ab$ and $\mathcal{A}=\mathcal{A}_A\wr \mathcal{A}_{G/A}$. So in view of Lemma~\ref{easy}, we may assume that $2\leq |Y|\leq p-2$. Since $\mathcal{A}$ is an $S$-ring,
$$\underline{Y}\cdot\underline{Y}=\underline{Yb}\cdot (\underline{Yb})^{-1}=|Y|e+\lambda\underline{A}^\#$$
for some positive integer $\lambda$. Therefore $D=Yb$ is a nontrivial difference set in~$A$. If $p=4q+1$, we obtain a contradiction to Theorem~\ref{main4}. Suppose that $p=2q+1$. Due to the above discussion and Lemma~\ref{dif2q}, every basic set of $\mathcal{A}$ outside $A$ has size~$q$ or~$q+1$. Therefore $\rk(\mathcal{A})=4$ and hence $\mathcal{A}=\mathcal{A}(D)$. Thus, Statement~$(4)$ of Theorem~\ref{classification} holds.
\end{proof}

\begin{proof}[Proof of Theorem~\ref{main3}]
Let $p$ be a Fermat prime or $p=4q+1$, where $q$ is prime. Let us prove that $\mathcal{A}$ is schurian. One of the statements of Theorem~\ref{classification} holds for $\mathcal{A}$. If Statement~$(1)$ or~$(2)$ holds, then $\mathcal{A}$ is obviously schurian. If Statement~$(3)$ holds, then  $\mathcal{A}$ is schurian because $C_{2p}$ is a Schur group by~\cite[Theorem~1.1]{EKP1}. If Statement~$(4)$ holds, then $p\in\{5,13\}$ and $\mathcal{A}$ is schurian by computer calculations~\cite{Ziv}.
\end{proof}

\section{Further discussion}

\subsection{$S$-rings from relative difference sets}

As it was mentioned in the introduction, the only non-Schur dihedral groups of order at most~$62$ which satisfy all the conditions from Theorem~\ref{main2} and Corollary~\ref{conditions} are $D_{52}$ and $D_{60}$. Nonschurian $S$-rings over these groups arise from relative difference sets. Let us discuss their constructions.

Let $A$ be a group and $A_0$ a proper nontrivial subgroup of $A$ of order and index~$n$ and $m$, respectively. A subset $D\subseteq A$ of size~$k$ is called a \emph{relative difference set} (\emph{RDS} for short) with forbidden subgroup~$A_0$ if 
\begin{equation}\label{rds}
\underline{D}\cdot \underline{D}^{-1}=ke+\lambda(\underline{A}-\underline{A_0})
\end{equation}
for some positive integer $\lambda$. The numbers $(m,n,k,\lambda)$ are called the \emph{parameters} of~$D$. Due to the definition of RDS, we have $|D\cap A_0a|\leq 1$ for every $a\in A$. Suppose that $A$ is abelian, $G$ is a generalized dihedral group associated with $A$, and $b\in G\setminus A$. Consider the partition $\mathcal{S}$ of $G$ into the following sets:
$$X_a=\{a\},~X_0=A\setminus A_0,~Y_a=Dab,~Y_0=(A\setminus DA_0)b,~a\in A_0.$$

\begin{lemm}\label{rdsring}
The $\mathbb{Z}$-module $\mathcal{A}=\mathcal{A}(D)=\Span_{\mathbb{Z}}\{\underline{X}:~X\in \mathcal{S}\}$ is an $S$-ring over $G$.
\end{lemm}

\begin{proof}
One can see that every set from $\mathcal{S}$ outside $A_0$ is inverse-closed and hence $X^{-1}\in \mathcal{S}$ for every $X\in \mathcal{S}$. Since $Y_0$ is a complement to the union of all other sets from $\mathcal{S}$, it suffices to verify that $\underline{X}\cdot \underline{Y}\in \mathcal{A}$ for all $X,Y\in \{X_a,X_0,Y_a:~a\in A_0\}$. A straightforward computation using Eq.~\eqref{rds} in the group ring $\mathbb{Z}G$ implies that
$$\underline{X_{a_1}}\cdot\underline{X_{a_2}}=\underline{X_{a_1a_2}},~\underline{X_{a}}\cdot\underline{X_{0}}=\underline{X_{0}}\cdot\underline{X_{a}}=\underline{X_0},~\underline{X_{a_1}}\cdot\underline{Y_{a_2}}=\underline{Y_{a_1a_2}},~\underline{Y_{a_2}}\cdot\underline{X_{a_1}}=\underline{Y_{a_1^{-1}a_2}},$$
$$\underline{X_0}^2=(mn-2n)\underline{X_0}+(mn-n)\underline{A_0},~\underline{X_0}\cdot\underline{Y_a}=\underline{Y_a}\cdot \underline{X_0}=k\underline{Ab}-\underline{Y_0},$$
$$\underline{Y_{a_1}}\cdot\underline{Y_{a_2}}=k\underline{X_{a_1a_2^{-1}}}+\lambda \underline{X_0},$$
for all $a,a_1,a_2\in A_0$. 
\end{proof}

From~\cite[Theorem~1.2]{Arasu} it follows that there exists an RDS $D$ with parameters 
$$\left(\frac{q^{d}-1}{q-1},n,q^{d-1},\frac{q^{d-2}(q-1)}{n}\right)$$
in a cyclic group, where $q$ is a prime power and $d,n\geq 2$ are integers, if and only if $n$ divides~$q-1$ when $q$ is odd or $d$ is even and $n$ divides~$2(q-1)$ when $q$ is even and $d$ is odd. Therefore one can construct an $S$-ring $\mathcal{A}(D)$ over a dihedral group of order~$2l$ for each $l$ of the form~$l=\frac{(q^{d+1}-1)n}{q-1}$, where $q$, $d$, and $n$ satisfy the above condition, in particular for $l=15$ and $l=26$. The computational results~\cite{Ziv} imply the following: $(1)$ $\mathcal{A}(D)$ is schurian if $l<26$; $(2)$ if $l=26$, then $\mathcal{A}(D)$ is a unique nonschurian $S$-ring over $D_{52}$; $(3)$ a nonschurian $S$-ring over $D_{60}$ can be obtained as a generalized wreath product of a schurian $S$-ring $\mathcal{A}(D)$ over $D_{30}$ and some other $S$-ring. It would be interesting to study a schurity of $\mathcal{A}(D)$ in a general case. Finally, it should be mentioned that the Cayley graph $\cay(G,Db)$, whose connection set is a basic set of $\mathcal{A}(D)$, is exactly a graph from~\cite[Proposition~5.1]{QDK} appeared in the context of studying $2$-arc transitive dihedrants (see also~\cite{Jin}).

\subsection{More on schurity of $D_{2p}$}

We do not know whether it is possible to extend Theorem~\ref{main3} to other primes~$p$. Recall that if $p\equiv 3\bmod 4$ and $p>11$, then $D_{2p}$ is not Schur by Lemma~\ref{3mod}. In this case, there is a Paley difference $D$ set in $C_p$ and the $S$-ring $\mathcal{A}(D)$ is nonschurian whenever $p>11$. However, $D_{2p}$ can be not Schur even if $p\equiv 1\bmod 4$ which follows from the lemma below.

\begin{lemm}\label{nonschur2}
If $p$ is a prime of the form $p=4t^2+1$ or $p=4t^2+9$, where $t\geq 3$ is an odd integer, then $D_{2p}$ is not Schur.
\end{lemm}

\begin{proof}
From~\cite[Theorem~VI.8.11(b)]{BJL} it follows that there is a biquadric residue difference set $D$ in a cyclic group of order~$p$ with parameters~$(v,k,\lambda)=(4t^2+1,t^2,\frac{t^2-1}{4})$ if $p=4t^2+1$ and~$(v,k,\lambda)=(4t^2+9,t^2+3,\frac{t^2+3}{4})$ if $p=4t^2+9$. Suppose that $p=4t^2+1$. Then $k=t^2$ divides $v-1=4t^2$. If $t>3$, then $1+\sqrt{k}=1+t>\frac{v-1}{k}=4$. So the design corresponding to~$D$ is not $2$-transitive by~\cite[Theorem~8.3]{Kantor}. Therefore the $S$-ring $\mathcal{A}(D)$ over $D_{2p}$ constructed from~$D$ as in Section~$3$ is nonschurian by~\cite[Lemma~5.2]{PV} and hence $D_{2p}$ is not Schur. If $t=3$, then it is easy to see that $v=4t^2+1=37$ can not be presented in the form $\frac{q^{d+1}-1}{q-1}$ for some prime power $q$  and $d\geq 2$ and hence $D_{2p}$ is not Schur by Lemma~\ref{schurdiff}.

Now let $p=4t^2+9$. Assume that $D_{2p}$ is Schur. Then Lemma~\ref{schurdiff} implies that $D$ must have parameters 
$$\left(\frac{q^{d+1}-1}{q-1},\frac{q^d-1}{q-1},\frac{q^{d-1}-1}{q-1}\right),$$
where $q$ is a prime power and $d\geq 2$. The straightforward computation yields that
$$\frac{k}{\lambda}=4=q+\frac{q-1}{q^{d-1}-1}.$$
Therefore $q=3$, $d=2$, and $t=1$, a contradiction to $t\geq 3$.  
\end{proof}

Observe that if $p=4t^2+9$, then $p-1$ can be of the form $p-1=4m$, where $m$ is odd and square-free. For example, $p=3373=4\cdot 29^2+9=4\cdot 3 \cdot 281+1$. It would be interesting to study schurity of $D_{2p}$ for all primes~$p$ such that $p\equiv 1\bmod 4$.

\subsection{Schurity of Frobenius groups}

Recall that a Frobenius group $C_p\rtimes C_{p-1}$ is not Schur for every prime $p\geq 7$ by Corollary~\ref{frobp}. It seems interesting to ask which Frobenius groups are Schur. Denote the set of all integers of one of the forms
$$p^k,~pq^k,2pq^k,~pqr,~2pqr,$$   
where $p$, $q$, and $r$ are primes and $k\geq 1$, by $\mathcal{P}$ and the set of all odd integers from $\mathcal{P}$ by $\mathcal{P}_0$. The proposition below provides several restrictions on the structure of a Frobenius Schur group.

\begin{prop}\label{frobenius}
Let $G$ be a Frobenius Schur group. Then $G$ is isomorphic to one of the following groups:
\begin{enumerate}

\tm{1} $C_{n}\rtimes C_{m}$, where $n\in \mathcal{P}_0$ and $m\in \mathcal{P}$;

\tm{2} $(E_4\times C_{p^k})\rtimes C_3$, $(E_4\times C_{pq})\rtimes C_3$, where $p$ and $q$ are odd primes and $k\geq 1$;

\tm{3} $E_4\rtimes C_3$, $E_{27}\rtimes C_{13}$, $E_{32}\rtimes C_{31}$.

\end{enumerate}

\end{prop}

\begin{proof}
We start the proof with the lemma below which follows from computational results~\cite{Ziv} and~\cite[Lemma~3.7]{PV}.

\begin{lemm}\label{small}
The Frobenius groups $E_8\rtimes C_7$, $E_9\rtimes C_2$, $E_{16}\rtimes C_3$, $E_{16}\rtimes C_5$, $E_{27}\rtimes C_2$ are not Schur.
\end{lemm}

Let $G=N\rtimes A$ be a Frobenius Schur group, where $N$ and $A$ are the Frobenius kernel and complement of $G$, respectively. From~\cite[Theorem~1.3]{PV} it follows that $G$ is metabelian and hence the derived subgroup $G^\prime$ is abelian. Since $G$ is Frobenius, $G^\prime=N\rtimes A^\prime$. If $A^\prime$ is nontrivial, then $G^\prime$ is also Frobenius and, in particular, nonabelian. So $A^\prime$ is trivial and consequently $N$ and $A$ are abelian. By Lemma~\ref{schursection}, the groups $N$ and $A$ are Schur. Due to~\cite[Corollary~6.17]{Isaacs}, $A$ is cyclic. Therefore~\cite[Theorem~1.1]{EKP1} yields the following

\begin{lemm}\label{complcycl}
The Frobenius complement $A$ is a cyclic group of order~$n\in \mathcal{P}$. 
\end{lemm}

Observe that $N$ does not have a characteristic subgroup of order~$2$. Together with~\cite[Theorem~3.3]{PV}, this implies the following

\begin{lemm}\label{abelschur}
The Frobenius kernel $N$ is isomorphic to one of the groups below:
\begin{enumerate}
\tm{1} a cyclic group of order $n\in \mathcal{P}_0$;

\tm{2} $E_n$, where $n\in\{4,8,9,16,27,32\}$;

\tm{3} $C_3\times C_{3^k}$, $E_9\times C_p$, where $k\geq 2$ and $p\neq 2$ is a prime;

\tm{4} $E_4 \times C_{p^k}$, $E_4\times C_{pq}$, $E_{16} \times C_p$, where $k\geq 1$ and $p\neq 2$ and $q$ are primes. 

\end{enumerate} 

\end{lemm}

If Statement~$(1)$ of Lemma~\ref{abelschur} holds for $N$, then Statement~$(1)$ of Proposition~\ref{frobenius} holds for $G$. Further, we assume that $N$ is noncyclic and hence one of Statements~$(2)$-$(4)$ of Lemma~\ref{abelschur} holds for $N$.

Suppose that $|N|$ is even. Then $P=\Syl_2(N)$ is nontrivial and $|A|$ is odd. From Lemma~\ref{abelschur} it follows that $P$ is isomorphic to one of the groups
$$E_4,~E_8,~E_{16},~E_{32}.$$
Since $P$ is characteristic in $N$, the group $H=P\rtimes A$ is Frobenius. By Lemma~\ref{schursection}, the group $H$ and each of its subgroups are Schur. Clearly, $|A|$ divides $|P|-1$. If $P\cong E_8$, then $H\cong E_8\rtimes C_7$, a contradiction to Lemma~\ref{small}. If $P\cong E_{16}$, then $H$ is isomorphic to one of the groups $E_{16}\rtimes C_3$, $E_{16}\rtimes C_5$, $E_{16}\rtimes C_{15}$. In all cases, $H$ has a subgroup isomorphic to $E_{16}\rtimes C_3$ or $E_{16}\rtimes C_5$, a contradiction to Lemma~\ref{small}. 

If $P\cong E_4$, then $N$ is isomorphic to one of the groups $E_4$, $E_4 \times C_{p^k}$, $E_4\times C_{pq}$ by Lemma~\ref{abelschur} and $A\cong C_3$. So Statement~$(2)$ or~$(3)$ of Proposition~\ref{frobenius} holds for $G$. If $P\cong E_{32}$, then $N\cong E_{32}$ by Lemma~\ref{abelschur} and $A\cong C_{31}$. In this case, Statement~$(3)$ of Proposition~\ref{frobenius} holds for $G$.

Now suppose that $|N|$ is odd. Then Lemma~\ref{abelschur} yields $N$ has a characteristic subgroup isomorphic to $E_9$ or $N\cong E_{27}$. In the first case, $G$ has a subgroup isomorphic to $E_9\rtimes C_2$ which is not Schur by Lemma~\ref{small}, a contradiction to Lemma~\ref{schursection}. In the second one, $A$ is isomorphic to one of the groups $C_2$, $C_{13}$, $C_{26}$ because $|A|$ divides $|N|-1=26$. If $A\cong C_2$ or $A\cong C_{26}$, then $G$ has a subgroup isomorphic to $E_{27}\rtimes C_2$ which is not Schur by Lemma~\ref{small}, a contradiction to Lemma~\ref{schursection}. If $A\cong C_{13}$, then $G\cong E_{27}\rtimes C_{13}$ and Statement~$(3)$ of Proposition~\ref{frobenius} holds for $G$. 
\end{proof}

It seems reasonable to ask about schurity of groups from Proposition~\ref{frobenius}. Computer calculations~\cite{Ziv} imply that all Frobenius groups of order at most~$62$ satisfying the conditions from Proposition~\ref{frobenius} are Schur except for the Frobenius group $C_{19}\rtimes C_3$ which is not Schur. A unique nonschurian $S$-ring over this group is a primitive $S$-ring of rank~$3$ arising from a strongly regular Cayley graph which is a block graph of a Steiner triple system. A construction of this $S$-ring is described in detail in~\cite{PR2}.

\end{document}